\documentclass[11pt]{article}
\usepackage{amsmath,amssymb,amsthm,color,graphicx}

\usepackage[unicode,breaklinks=true,colorlinks=true]{hyperref}

\newcommand{\cmt}[1]{{\color[rgb]{0,0.3,1} #1}}
\newcommand{\cmtr}[1]{{\color[rgb]{1,0,0} #1}}

\newcommand{\cmtm}[1]{{\mbox{}\color{magenta} #1}}
\definecolor{sienna}{rgb}{0.53, 0.18, 0.09}
\newcommand{\cmts}[1]{{\mbox{}\color{sienna} #1}}

\renewcommand{\cmt}[1]{#1}
\renewcommand{\cmtr}[1]{#1}

\renewcommand{\cmtm}[1]{#1}
\renewcommand{\cmts}[1]{#1}

\numberwithin{equation}{section}
\newtheorem{theorem}{Theorem}[section]

\newtheorem{lemma}[theorem]{Lemma}
\newtheorem{proposition}[theorem]{Proposition}

\theoremstyle{remark}
\newtheorem{remark}[theorem]{Remark}
\newcommand{\bke}[1]{\left ( #1 \right )}
\newcommand{\bkt}[1]{\left [ #1 \right ]}
\newcommand{\bket}[1]{\left \{ #1 \right \}}
\newcommand{\norm}[1]{ \| #1  \|}

\newcommand{\bka}[1]{{\langle #1 \rangle}}
\newcommand{\abs}[1]{\left | #1 \right |}

\newcommand\al{\alpha}
\newcommand\be{\beta}
\newcommand\ga{\gamma}
\newcommand\de{\delta}

\newcommand\ve{\varepsilon}
\newcommand\e {\epsilon}

\renewcommand\th{\theta}
\newcommand\ka{\kappa}
\newcommand\la{\lambda}
\newcommand\si{\sigma}
\newcommand\ph{\varphi} %

\newcommand\om{\omega}
\newcommand\Ga{\Gamma}
\newcommand\De{\Delta}

\newcommand\Si{\Sigma}
\newcommand\Om{\Omega}
\newcommand{\R}{\mathbb{R}}

\newcommand{\N}{\mathbb{N}}
\renewcommand{\div}{\mathop{\rm div}\nolimits}

\newcommand{\supp} {\mathop{\mathrm{supp}}}

\newcommand{\calK}{{ \mathcal K  }}

\newcommand{\pd}{\partial}
\newcommand{\nb}{\nabla}
\newcommand{\td}{\tilde}

\newcommand{\wbar}[1]{\overline{\rule{0pt}{2.4mm} {#1}}}
\newcommand{\lec}{{\ \lesssim \ }}

\newcommand{\I}{\infty}

\newcommand{\bs}{\backslash}
\renewcommand{\[}{\begin{equation}}
\renewcommand{\]}{\end{equation}}
\newcommand{\EQ}[1]{\begin{equation}\begin{split} #1 \end{split}\end{equation}}
\begin{document}

\title{Green tensor of the Stokes system and asymptotics of
stationary Navier-Stokes flows in the half space}

\author{Kyungkeun Kang%
\thanks{Department of Mathematics, Yonsei University, Seoul 120-749,
South Korea. Email: \texttt{kkang@yonsei.ac.kr}}
\and
Hideyuki Miura%
\thanks{Department of Mathematical and Computing Sciences, Tokyo Institute of Technology, Tokyo 152-8551, Japan. Email: \texttt{miura@is.titech.ac.jp}}
\and
Tai-Peng Tsai%
\thanks{Department of Mathematics, University of British Columbia,
Vancouver, BC V6T 1Z2, Canada. Email:  \texttt{ttsai@math.ubc.ca}
}}
\date{\jobname .tex \quad \today}

\maketitle

\begin{abstract}

  We  derive refined estimates of the Green tensor
  of the stationary Stokes system in the half space. We then
  investigate the spatial asymptotics of stationary solutions of the
  incompressible Navier-Stokes equations in the half space. We also discuss the asymptotics of fast decaying flows
  in the whole space and exterior domains. In the Appendix
  we consider axisymmetric self-similar solutions.


\end{abstract}

\section{Introduction} \label{sec1}
We are concerned with the Stokes system
in the $n$-dimensional half space $\R^n_+$, $n \ge 2$,
\begin{equation}\label{eq1-1}
\tag{S}
\begin{cases}
-\Delta u+\nabla q=f+\nabla\cdot F, \qquad{\rm div}\,u=0\qquad
&\mbox{in }\,\,\R^n_+,
\\
\qquad\qquad\qquad\qquad u=0\qquad &\mbox{on }\,\, \pd  \R^n_+,
\end{cases}
\end{equation}
or of the Navier-Stokes equations
\begin{equation}\label{eq1-2}
\tag{NS}
\begin{cases}
-\Delta u+ (u\cdot\nabla )u+\nabla p=f+\nabla\cdot F, \qquad{\rm
div}\,u=0\qquad &\mbox{in }\,\,\R^n_+,\\
\qquad\qquad\qquad\qquad u=0\qquad &\mbox{on }\,\, \pd  \R^n_+.
\end{cases}
\end{equation}
Above $u=(u_i)_{i=1}^n :\R^n_+ \to \R^n$ is the velocity field, $p : \R^n_+ \to \R$ is the pressure, and
$(f+ \nabla\cdot F)_i = f_i + \pd_j F_{ji}$ is the given force. We denote
\[
\R^n_+=\bket{x=(x',x_n):\ x'=(x_1,\ldots,x_{n-1})\in \R^{n-1}, \
  x_n>0},
\]
with boundary $\Sigma=\pd \R^n_+= \bket{x=(x',x_n):\ x'\in \R^{n-1}, \
x_n=0}$. Denote
\[
x^* = (x',-x_n) \quad \text{if}\quad x=(x',x_n).
\]

The purpose of this paper is to study the asymptotic
behavior of the Navier-Stokes flows for small forces. To this
end, we also derive pointwise estimates of the Green tensor for the
Stokes system \eqref{eq1-1}.
Our linear results are valid for dimension $n \ge 2$, while our nonlinear results are mostly for $n \ge 3$. 

\subsection{Background and motivation}

As  shown by Lorentz \cite{Lorentz} (see also \cite{Oseen,Galdi},
\S\ref{S2.1}), the fundamental solution
$\{U_{ij}(x)\}_{i,j=1,\ldots,n}$ of the Stokes system in the whole
space $\R^n$ has the same decay properties as that for the Laplace
equation, namely (for $n \ge 3$)
\[\label{eq1-4}
|U_{ij}(x)| \lec |x|^{2-n}.
\]
(We denote $A \lec B$ if there is some constant $C$ so that $A \le
CB$.) As a result, when the force is small (of order $\e$) and
sufficiently localized (i.e. the force decays sufficiently fast),
one can construct the solutions to the Navier-Stokes equations with
the same decay
\[
\label{eq1-6}
|u_{i}(x)| \lec \e\bka{x}^{2-n}, \quad \bka{x} := (\cmtm{2}+|x|^2)^{1/2}.
\]
By a standard cut-off argument, one can get solutions with the same
decay in an exterior domain (see \cite{Finn}).

However, when the domain is the half space $\R^n_+$ with no-slip
boundary condition, the Green tensor
$\{G_{ij}(x,y)\}_{i,j=1,\ldots,n}$ to \eqref{eq1-1} has a faster
decay rate than \eqref{eq1-4},
\[
\label{eq1.5} |G_{ij}(x,y)| \lec |x|^{1-n}, \quad (|y|\le 1\ll |x|),
\]
(see Section \ref{sec2} for detailed review),
and one can construct solutions to \eqref{eq1-2} with
the same decay (see e.g.~\cite{ChangHuakang}, \cite{Galdi})
\[
\label{eq1-8}
|u_{i}(x)| \lec \e\bka{x}^{1-n}
\]
for small localized forces. 

This project starts with the following
intuition: For fixed $|y| \lec 1$ (corresponding to localized force), the decay of $G_{ij}(x,y)$ in $x$ should be similar to the
Poisson kernel of \eqref{eq1-1}.
It has been shown by  Odqvist
\cite[\S2]{Odqvist} (see \S \ref{S2.2})
 that the  Poisson tensor of \eqref{eq1-1} is
\[
K_{ij}(x)= \frac{2 x_nx_ix_j}{\om_n |x|^{n+2}},
\]
where $\om_n=\frac {2 \pi^{n/2}}{n \Ga(n/2)}$
 is the volume of the unit ball in $\R^n$.
Thus we expect that
\[
\label{eq1-11}
|G_{ij}(x,y)| \lec \frac{x_n}{ |x|^{n}}, \quad (|y|\lec 1 \ll |x|).
\]
For $x_n \sim |x|$, this estimate reduces to \eqref{eq1.5}, while it
implies more decay than \eqref{eq1.5} for $x_n \ll |x|$. As a
result, the Navier-Stokes flow for a small localized force is
expected to have the same decay as the Green tensor. The goal of
this paper is justify this intuition and identify the leading
asymptotic profile of solutions of  \eqref{eq1-2} with small
localized force.

\subsection{Main results}

Section \ref{sec2} is concerned with the refined upper bounds for the Green
tensor and its derivatives of the Stokes system in $\R^n_+$ for $n \ge 3$ and $n=2$. In particular, when $n\ge 3$, for $x,y \in \R^n_+$ we have
\[
\label{th1.1-1} |G_{ij}(x,y)|\le \frac {C_0
x_ny_n}{|x-y|^{n-2}|x-y^*|^{2}},\qquad i,j\in\{1,\ldots,n\},
\]
where the constant $C_0>0$ is independent of $x,y\in \R^n_+$, and recall $y^*
= (y',-y_n)$ for $y=(y',y_n)$.
Furthermore, when $j=n$, the estimate \eqref{th1.1-1} can be
improved as
\[
\label{th1.1-2} |G_{in}(x,y)|\le \frac {C_0 x_n
y_n^2}{|x-y|^{n-2}|x-y^*|^{3}}.
\]
The above estimates justify \eqref{eq1-11} and imply extra decay when $j =n$ and $\abs{y}\ll |x|$.
 \cmts{See Theorems \ref{th2.2} for the above estimates, and \eqref{eq:intro-2-2-1} and Theorem \ref{th2.5b} for refined gradient estimates.}

In Section \ref{sec3}, we identify the leading profile
of the Navier-Stokes flows in $\R^n_+$, $n\ge 3$, for small localized forces.
To be more precise, suppose that $\abs{f(x)}\lesssim \ve\langle x
\rangle^{-a}$ and $\abs{F(x)}\lesssim \ve\langle x \rangle^{-a+1}$
with $a \in (n+1,n+2)$ for sufficiently small $\ve>0$. Then, there
exists a unique solution $(u, p)$ of the Navier-Stokes equations
\eqref{eq1-2} with $\abs{u(x)}\lesssim \frac{ \epsilon
x_n}{\bka{x}^n}$ and, furthermore, its asymptotics is given as
\begin{equation}\label{eq1.2-1}
u_i(x)=\sum_{j=1}^{n-1}K_{ij}(x)\tilde{b}_j + O\left(\frac{\ve
x_n}{\langle x \rangle^{a-1} }\right),
\end{equation}
where
\begin{equation}\label{eq1.2-2}
\tilde{b}_j = \int_{\R^n_+} \bket{ u_n(y)u_j(y) +y_n
f_j(y)-F_{nj}(y)} dy, \quad (j <n).
\end{equation}
Here, for simplicity, we assume that $a\ \in (n+1,n+2)$ but it
suffices to restrict $a>n+1$ (see Theorem \ref{nse-200} for the
details). On the other hand, for any given small numbers
$\tilde{b}_1, \tilde{b}_2,\cdots, \tilde{b}_{n-1}$ we construct a
solution of the Navier-Stokes equations satisfying \eqref{eq1.2-1}
and \eqref{eq1.2-2} (consult Theorem \ref{nse-600}). For the Stokes
system, we present similar formulas including two dimension for fast decaying $f$ and $F$ without smallness assumption (see
Theorem \ref{stokes-100}).

In vector form, with $(\vec K_j)_i = K_{ij}$, \eqref{eq1.2-1}
reads
$
{u(x) = \sum_{j=1}^{n-1} \vec K_{j}(x)\tilde{b}_j +
\textrm{error}}.
$
Thus the leading asymptotic of the solution is given by a
linear combination of $\vec K_1 ,\ldots, \vec K_{n-1}$. That $\vec K_n$ is not present is because a solution of \eqref{stokes-KMT-10}-\eqref{stokes-KMT-20} should have zero flux on any hemisphere $S_R^+= \bket{ x\in \R^n_+, |x|=R}$, while $\vec K_n$  has nonzero flux.

To derive \eqref{eq1.2-1}, it is required to estimate the
derivatives of the Green tensor.
However it is not an easy task, as the formulas for the Green tensor
span more than one full page in the literature (see
\cite[Appendix 1]{MR548252} for $n=2,3$, and \cite[IV.3]{Galdi} for
higher dimensions). Fortunately, we are able to refine the approach
of \cite[Appendix 1]{MR548252} and derive estimates for
derivatives of $G_{ij}$ for $n\ge 2$,
\begin{equation}\label{eq:intro-2-2-1}
\abs{\nb_x^\al \nb_y ^\be G_{ij}(x,y)}\le  \frac{ C_m
x_n}{|x-y|^{n-2+m}|x-y^*|}
\end{equation}
for any multi-indices $\al$ and $\be$ with $|\al|+|\be|=m>0$ and
$\al_n=0$ (see Theorem \ref{th2.5b}). 
\cmt{We emphasize
that the factor $x_n$ in \eqref{eq:intro-2-2-1} is lost only if
$\al_n>0$ and differentiations  in the $y$ variable
does not kill the $x_n$ factor in \eqref{eq:intro-2-2-1}. This is important for the refined error
estimates, which contain the $x_n$ factor, in \eqref{eq1.2-1} and Theorem \ref{nse-200}.}

\cmts{As applications, we consider the asymptotics of \emph{general} solutions in $\R^n_+$ in Theorem \ref{application-KMT-100} under various smallness assumptions on the forces or the solutions, and we also consider similar questions when we further remove the boundary condition in a neighborhood of the origin in Theorem \ref{nse-500}. 
The latter} turns out to be a type of aperture
problem and we recover previously known asymptotic profiles of
solutions with a refined decay estimate for error terms (see Theorem
\ref{nse-500} for the details and compare with \cite{BP1992} and
\cite{Galdi}).

\cmts{In Section \ref{sec4},
we extend the methods of Section \ref{sec3} and 
study the asymptotics of \emph{fast decaying} solutions of the
Navier-Stokes equations in the \emph{whole space and exterior domains} in $\R^n$, $n \ge 3$,
where by fast decaying solution we mean 
a solution which decays faster than the fundamental solution, usually due to cancellation.}
For general small localized forces, solutions are
expected to decay like \eqref{eq1-6}. For example, in case of three
dimensional exterior domains, it was shown in \cite{NP00} that
leading asymptotic of the solution is a minus one homogeneous
profile, which is nothing but one of the Slezkin-Landau solutions of
(NS) (see \cite{Korolev-Sverak}).
However, if we assume further certain cancelation of the force, one may
expect an extra decay such as \eqref{eq1-8}. Indeed,
we prove that for such a case the solutions satisfy the decay \eqref{eq1-8}
and, in addition, their asymptotics are given by
\[
u(x) = b_0 \nb E(x)+\sum_{(k,j) \not = (n,n)} a_{jk}\Phi^{jk} (x) +
o(|x|^{1-n})
\]
for some constants $b_0$ and $a_{jk}$, where $ \Phi^{jk}_i = \pd_k
U_{ij}$ and $E$ is the fundamental solution of Laplace equation (see Proposition \ref{thm:exterior1},
Theorems \ref{thm:whole2} and \ref{thm:exterior2} for the details).

Finally in the Appendix we study the nonexistence of axisymmetric
self-similar solutions of \eqref{eq1-2} in $\R^3_+$ under suitable
boundary conditions. It is relevant to the asymptotic problem since their existence would be an obstacle to proving \eqref{eq1.2-1} which has faster decay than self-similar solutions.

In this paper we do not consider the asymptotic formula for 
two dimensional Navier-Stokes equations, for which we do not know a general existence theory of solutions satisfying the decay \eqref{eq1-8} even in the whole
space, 
because the nonlinear term does not have enough decay. To get existence for dimension two,
one usually needs either some symmetry assumptions on the forces (and hence the solutions, see 
e.g.~\cite{GPS} for aperture problems,  \cite{Yamazaki09} for the whole space, and
\cite{Yamazaki11} for exterior domains),
or 
the solutions have to be close to some special flows to ensure that the solutions decay sufficiently fast; see e.g.~\cite{HW}. 

After a preprint of this paper was posted to arXiv (arXiv:1606.01854v1), Professor D.~Iftimie kindly informed us that a formula similar to \eqref{eq1.2-1} for dimension three,
with the asymptotic profile spanned by the
Poisson kernel only, also appeared in the thesis of Dr.~A.~Decaster \cite[Remark 4.2.4]{Decaster-thesis}, with the proof in its Section 4.4. Our error estimate is more refined due to our new Green tensor estimates.

\section{Green tensor of the Stokes system in the half space} \label{sec2}

In this section we derive refined estimates of the Green tensor of the stationary Stokes system in the half space
$\R^n_+$, $n \ge 2$. We first recall in \S\ref{S2.1} the Lorentz tensor, which is the fundamental solution of the stationary Stokes system  in $\R^n$. We then recall in \S\ref{S2.2} the Odqvist tensor, which is the Poisson kernel of the stationary Stokes system  in $\R^n_+$. We finally study in \S\ref{S2.3} the Green tensor.

Let $n \ge 2$ and
$E(x)$ and $\Phi(x) = \Phi(|x|)$ be the fundamental solutions of the
Laplace and biharmonic equations in $\R^n$,
\begin{equation}
\label{eq2.1}
-\De E = \de,\quad \De^2 \Phi =\de,
\end{equation}
where $\de$ is the Dirac delta function. 
Recall
\begin{equation}
E(x) = 2\ka |x|^{2-n} \quad (n \ge 3); \quad E(x) =- 2\ka \log {|x|} \quad (n=2),
\end{equation}
where $\ka = \frac {1 }{2n(n-2)\om_n}$ if $n\ge 3$ and $\ka = \frac 1{4\pi}$ if $n=2$,
$\om_n = |B_1^{\R^n}|=\frac {2 \pi^{n/2}}{n \Ga(n/2)}$, and
$\nb E = - \frac x{n\om_n |x|^{n}}$ for all $n \ge 2$.
We can integrate $\pd_r (r^{n-1} \Phi') = -r^{n-1} E$ to get an explicit formula for $\Phi$:
\EQ{
\Phi(x) = \frac {|x|^2}{8\pi}( \log & |x| - 1) \quad (n=2); \qquad  %
\Phi(x) = -\ka \log |x| \quad (n=4);
\\
&\Phi(x) = \frac {\ka}{(n-4) } |x|^{4-n}\quad (n=3  \text{ or } n\ge 5).
}

\subsection{Lorentz tensor $U_{ij}$ in $\R^n$}
\label{S2.1}

The Lorentz tensor is the fundamental solution of the Stokes system in
$\R^n$, $n \ge 2$, (Lorentz \cite{Lorentz},  see \cite{Oseen} and \cite[\S IV.2]{Galdi}).
The Lorentz tensor $\vec U_j(x) = (U_{ij}(x))_{i=1}^n$ and $q_j(x)$
satisfy, for each fixed $j=1,\ldots,n$,
\[
\label{eq2.3}
-\De \vec U_{j} + \nb q_j = \de e_j, \quad \div \vec U_{j} =0, \quad
(x\in\R^n).
\]
Above $e_j$ is the unit vector in $x_j$ direction. 
 Component-wise,
\[
-\De U_{ij} + \pd_i q_j = \de \de_{ij}, \quad \pd_i U_{ij} =0, \quad
(x\in\R^n).
\]
Taking $\div$ of the first equation of \eqref{eq2.3}, we get
$\De q_j= \pd_j \de$ in the sense of distributions.  In view of \eqref{eq2.1}, we can take
$q_j = -\pd_j E$.  Thus $-\De U_{ij} = \de_{ij} \de + \pd_i \pd_j E$,
and we can take
\begin{equation}
\label{eq2.5} U_{ij} = (-\de_{ij} \De + \pd_i \pd_j)\Phi = \de_{ij} E
+ \pd_i \pd_j\Phi, \quad q_j = -\pd_j E.
\end{equation}

For dimension $n=2$, we have
\[
\label{Uij.def.n2}
U_{ij}(x) =\frac {1}{4\pi} \bkt{- \de_{ij}\log |x|+ \frac {x_i
    x_j}{|x|^2} }, \quad q_j(x) = \frac {x_j}{2\pi|x|^2}.
\]

For dimension $n\ge3$, we have
\[
\label{Uij.def.n3}
U_{ij}(x) = \frac 1{2n(n-2)\om_n} \bkt{ \frac {\de_{ij}}{|x|^{n-2}} +
  (n-2)\frac {x_i x_j}{|x|^n} }, \quad q_j(x) = \frac
{x_j}{n\om_n|x|^n}.
\]

Summarizing, for $n \ge 2$,
\[
\label{Uij.def}
U_{ij}(x) = \frac {1}{2}\de_{ij} E(x) +\frac{1}{2 n\om_n} \frac {x_i
  x_j}{|x|^n} , \quad q_j(x) = \frac {x_j}{n\om_n|x|^n}.
\]

\subsection{Odqvist tensor $K_{ij}$ in  $\R^n_+$}
\label{S2.2}
The Odqvist tensor $K_{ij}$ is the Poisson kernel for the Stokes
system in the half space $\R^n_+$, $n \ge 2$. A solution $(u,p)$ of the
homogeneous Stokes system in the half space $\R^n_+$ with boundary
data $ \phi:\Si=\pd \R^n_+\to\R^n$ is given by
\[
u_i(x) = \int_{\Si} K_{ij}(x-z)\phi_j (z) dz, \quad p(x) = \int_{\Si}
k_j(x-z)\phi_j (z) dz,
\]
where
\[
\label{Kij.def}
K_{ij} = 2(\pd_n U_{ij} +  \pd_j U_{in} - \de_{jn}q_i), \quad k_j = - 4 \pd_j q_n .
\]
One computes directly using \eqref{Kij.def}, \eqref{Uij.def} and
\eqref{Uij.def.n2} to get, for $n \ge 2$,
\[
K_{ij}(x)= \frac{2 x_nx_ix_j}{\om_n |x|^{n+2}}, \quad k_j (x)= -\pd_j
\frac{4 x_n}{n\om_n |x|^{n}}.
\]
One can verify that, when $x_n>0$, using \eqref{Kij.def}, $\De U_{ij} = \pd_i q_j =
- \pd_{ij}E$, and $\pd_i U_{ij}=0$,
\EQ{
-\De K_{ij} + \pd_i k_j &= 2(-\pd_n \pd_{ij}E - \pd_j \pd_{in}E - 0)+
4 \pd_{ij}\pd_nE =0, \\ \pd_i K_{ij} &= 2(\pd_n \pd_i U_{ij} +
\pd_{j}\pd_i U_{in} + \de_{jn} \De E) =0.
}
One can also verify that, for $ \phi\in C^1_c(\Si;\R^n)$,
\[
\int_{\Si} K_{ij}(x-z)\phi_j (z) dz \to \phi_i(x')\quad \text{as } x_n
\to 0_+.
\]

The above is derived by Odqvist
\cite[\S2]{Odqvist} using double layer potentials, see also \cite[\S IV.3]{Galdi}.
One may also implicitly derive $K_{ij}$  using Fourier transform
in $x'$ as in  Solonnikov \cite{MR0415097}, see also Maekawa-Miura \cite{Maekawa-Miura}.

\subsection{Green tensor $G_{ij}$ in  $\R^n_+$}
\label{S2.3}
For the Stokes system in the half space $\R^n_+$, $n \ge 2$,  the Green tensor $\vec G_j(x,y) =
(G_{ij}(x,y))_{i=1}^n$ and $g_j(x,y)$, for each fixed $j=1,\ldots,n$ and $y \in \R^n_+$, satisfy
\[
-\De_x \vec G_{j} + \nb_x g_j = \de_y e_j, \quad \div_x \vec G_{j} =0,
\quad (x\in\R^n_+),
\]
\[
 \vec G_{j}(x,y)|_{x_n=0}=0.
\]
In components,
\[
\label{eq2.11}
-\De_x G_{ij} + \pd_{x_i} g_j = \de (x-y)\de_{ij}, \quad \pd_{x_i} G_{ij}
=0, \quad (x\in\R^n_+),
\]
\[
\label{eq2.12}
G_{ij}(x,y)|_{x_n=0}=0.
\]

Denote
\[
y^* = (y',-y_n) \quad \text{if }y=(y',y_n); \quad \e_j = 1 -2 \de_{nj}.
\]
Thus $y^*_j = \e_j y_j$. By \eqref{Uij.def}, if $i=j$, then $U_{ii}$
is even in all $x_k$. If $i\not =j$, $U_{ij}$ is odd in $x_i$ and
$x_j$, but even in $x_k$ if $k\not =i,j$. In particular, with $k=n$,
\[
\label{eq2.19}
U_{ij}(x^*)=\e_i\e_jU_{ij}(x).
\]

Let
\[
\label{tildeG.def}
\tilde G_{ij}(x,y)=U_{ij}(x-y)-\e_j U_{ij}(x-y^*).
\]
At $x_n=0$, with $z =(x',0)-y$,
\[
\tilde G_{ij}(x,y )|_{x_n=0} = U_{ij}(z  )-\e_j U_{ij}(z ^* )=
(1-\e_i)U_{ij}(z )
\]
by \eqref{eq2.19}. Thus
\[
\tilde G_{ij}(x,y)|_{x_n=0} = 0\quad(i<n);\quad \tilde
G_{nj}(x,y)|_{x_n=0} =2U_{nj}(x'-y',-y_n).
\]

We can now decompose
\[
\label{Gij.dec}
G_{ij}=\tilde G_{ij}+W_{ij},
\]
where $W_{ij}$ is given by the boundary layer integral
\[
\label{Wij.def}
W_{ij}(x,y)=-2\int_{\Si}K_{in}(x-\xi )U_{nj}(\xi -y)d\xi .
\]

\begin{lemma}
\label{th2.1}Fix $n \ge 3$.

(i) $G_{ij}(x,y)=G_{ji}(y,x)$;\quad

(ii) $G_{ij}(x,y)=\la^{n-2} G_{ij}(\la x,\la y)$ for any $\la>0$.
\end{lemma}
\begin{proof}
(i) The three dimensional  case can be found in Odqvist \cite[p.~358]{Odqvist}. The
  higher dimensional case is similar: For $x,y\in \R^n_+$,
let $\Om_\e = \R^n_+ \bs ( {B_\e(x)}\cup {B_\e(y)})$ for $0<\e \ll 1$. Put the second argument as superscript, e.g., $G(z,x)=G^x(z)$. One has
\EQ{
0 &= \lim_{\e \to 0_+} \sum_k \int_{\Om_\e}
\bket{ G_{ki}^x[\De_z G_{kj}^y-\pd_{z_k}g_j^y] -G_{kj}^y[\De_z G_{ki}^x-\pd_{z_k}g_i^x]  }dz
\\
&= \lim_{\e \to 0_+} \sum_k \int_{ {\pd B_\e(x)}\cup {\pd B_\e(y)}}
\bket{ G_{ki}^x[\nb_z G_{kj}^y-e_k g_j^y] -G_{kj}^y[\nb_z G_{ki}^x-e_k g_i^x]  }\cdot \nu
\\
&
=[0-G_{ji}(y,x)] -[ - G_{ij}(x,y)+0].
}
We have used the 
cancellation of $\nb_z G^x_{ki} \cdot \nb_z G^y_{kj}$.
We have also used \eqref{eq2.11}, \eqref{eq2.12} and the decay at infinity of $G_{ij}$.
 
(ii) It follows from \eqref{Gij.dec}, \eqref{Wij.def}, and the scaling properties of $U_{ij}$ and $K_{ij}$.
\end{proof}

In the following we derive an explicit formula for $G_{ij}$, following
the approach of Maz'ja, Plamenevski{\u\i}, and Stupjalis
\cite[Appendix 1]{MR548252} for $n=2,3$. See \cite[IV.3]{Galdi} for formulas for higher dimensions.
However, our formula is much more compact, and is suitable
for estimates.

\begin{theorem}
\label{th2.2a} Fix $n \ge 2$. For $x,y \in \R^n_+$, denote $w=x-y$, $z=x-y^*$,  $\th = \frac {x_ny_n}{|z|^{2}}\in (0,\frac14]$,
\EQ{
\label{Qs.def}
Q_s =  \frac 1{|w|^{s}} - \frac 1{|z|^{s}} - \frac{2s x_ny_n}{  |z|^{s+2}},\quad (s>0);\quad
Q_0 =-\log |w|+ \log |z| - \frac{2 x_ny_n}{  |z|^{2}}.
}
 Then
\EQ{
\label{Gij:formula}
G_{ij}(x,y)
={\de_{ij}} \ka Q_{n-2} + \frac1{2n\om_n} w_iw_j Q_n +
\frac{x_ny_n (w_iw_j+z_i\e_j z_j)}{\om_n |z|^{n+2}} .
}
\end{theorem}

{\it Remark.} Recall
 $\ka = \frac {1 }{2n(n-2)\om_n}$ if $n\ge 3$ and $\ka = \frac 1{4\pi}$ if $n=2$.
Since
\[
\label{eq2.29}
|w|^2 = |x'-y'|^2 + x_n^2+y_n^2 - 2 x_n y_n = |z|^2 - 4x_n y_n =
|z|^2(1-4\th),
\]
 $Q_s$ is the remainder of the first order Taylor expansion of  $|w|^{-s} =|z|^{-s}(1-4\th)^{-s/2}$ when $\th \ll 1$ for $s>0$, and similarly for $Q_0$ as $\log |w| - \log |z| =\frac 12 \log (1-4\th)$. We need $Q_0$ only if $n=2$.
The definition of $Q_s$ is not continuous in $s$ as $s\to 0_+$. In fact, 
$\frac 1s Q_s \to Q_0$ as $s \to 0_+$.
This discrepancy is related to the choices of the coefficient $\ka$ for $n=2$ and $n \ge 3$.

\begin{proof} Recall $G_{ij} = \td G_{ij} +W_{ij}$.
By \eqref{Uij.def} we may rewrite
\EQ{
 \label{EQ2.30}
\td G_{ij}(x,y)
&=\frac12{\de_{ij}}(E(w)-\e_j E(z))+ \frac1{2n\om_n} \bke{\frac{w_iw_j}{|w|^n} -\frac{\e_j z_iz_j}{|z|^n}  }
\\
&=\frac12{\de_{ij}}\bke{ 2\ka Q_{n-2} +2\de_{jn}E(z) +\frac {2x_ny_n}{n \om_n |z|^n}}
\\
&\quad + \frac1{2n\om_n} \bke{w_iw_j Q_n + \frac{w_iw_j - \e_j z_iz_j}{|z|^n}  +\frac {2nx_ny_nw_iw_j}{ |z|^{n+2}}}.
}
Above we have used $2\ka\cdot 2(n-2) = \frac {2}{n \om_n}$ for $n\ge 3$ and $2\ka\cdot 2= \frac {2}{n \om_n}$ for $n=2$.

To compute $W_{ij}$ defined by \eqref{Wij.def}, we will use the identity that, for $x,y \in \R^n_+$, $n \ge 2$,
\[
\label{Mazya16}
\int_{\Si} P(x-\xi) E(\xi-y) d\xi = E(x-y^*), \quad P(x)=\frac {2x_n}{n \om_n |x|^n}.
\]
It is because $P(x)$ is the Poisson kernel of
the Laplace equation in $\R^n_+$, while $E(x-y^*)$
is the unique bounded (or sublinear if $n=2$) harmonic function in
$ \R^n_+$ with the boundary value $E(x-y)|_{x_n=0}$ for fixed $y$.
Note
\[
K_{in}(x-\xi) = \frac {2x_n^2 (x_i - \xi_i)} {\om_n |x-\xi|^{n+2}} = \bkt{ \de_{in} -x_n \frac{\pd}{\pd x_i}} P(x-\xi) ,
\]
and
\[
U_{nj}(\xi -y) =\frac 12 \bkt{\de_{nj} - y_n \frac{\pd }{\pd y_j} } E(\xi-y).
\]
Thus, using  \eqref{Wij.def} and \eqref{Mazya16},
 \begin{align}%
W_{ij}(x,y) &= - \bke{ \de_{in}-x_n \frac{\pd}{\pd x_i} } \bke{\de_{nj} - y_n \frac{\pd }{\pd y_j} }
\int_\Si P(x-\xi) E(\xi-y) \,d\xi
\\
&=  - \bke{ \de_{in}-x_n \frac{\pd}{\pd x_i} } \bke{\de_{nj} - y_n \frac{\pd }{\pd y_j} }
E(x-y^*).
\label{formula-wij}
\end{align}
Expanding the derivatives, with $z=x-y^*$,
\EQ{
\label{EQ2.38}
W_{ij}(x,y)
&=-\de_{in}\de_{jn}E(z)
-\frac{\de_{in}y_n (-\e_jz_j)}{n \om_n |z|^n}
-\frac{\de_{jn}x_n z_i}{n \om_n |z|^n}
\\
& \quad +  \frac 1{n\om_n} x_ny_n \bke{-\frac {\de_{ij}\e_j }{|z|^n}
+ n \frac{z_i\e_j z_j}{|z|^{n+2}} }.
}
Summing \eqref{EQ2.30} and  \eqref{EQ2.38}, and cancelling $\de_{jn}\de_{ij}E(z)$, we get
\EQ{
G_{ij}(x,y)
&={\de_{ij}} \ka Q_{n-2} + \frac1{2n\om_n} w_iw_j Q_n +
\frac{x_ny_n (w_iw_j+z_i\e_j z_j)}{\om_n |z|^{n+2}}  + \frac R{2n\om_n |z|^n}
}
with
\EQ{
R &=   2\de_{ij}x_ny_n + (w_iw_j-\e_j z_iz_j)
+2\de_{in}y_n \e_jz_j
-2\de_{jn}x_n z_i  - 2  \de_{ij}\e_j x_ny_n  =0.
}
The above shows \eqref{Gij:formula}.
\end{proof}

We next estimate $G_{ij}$. For this purpose, it is useful to know the geometry of the level sets of $\th = \frac {x_ny_n}{|z|^{2}}\in (0,\frac14]$.
For fixed $y\in\R^n_+$ and $c \in (0,\frac 14)$, the region $\th \ge c$ corresponds to a closed disk
\[
D_c=\bket{(x',x_n) :\quad  |x'-y'|^2 + (x_n - (\frac 1{2c}-1)y_n) ^2 \le \frac {1-4c}{4c^2} \, y_n^2},
\]
which is inside $\R^n_+$, increases as $c$ decreases, $\cap_ {0<c<1/4} D_c = D_{1/4}= \{ y\}$ and $\cup_{0<c<1/4} D_c = \R^n_+$.
We also have
\begin{align}
\label{zsim1}
C ^{-1}y_n<|z| < C y_n &\quad \text{if}\quad \frac 1{10}\le \th\le \frac 14,
\\
\label{zsim2}
|w|<|z|< C |w| &\quad \text{if}\quad 0<\th\le \frac 1{10},
\end{align}
for some constant $C$ independent of $x,y \in \R^n_{+}$. Estimate \eqref{zsim1} is because that the radius of $D_c$ is $C(c)y_n$, while
\eqref{zsim2} follows from
\eqref{eq2.29}.
 For different $y\in\R^n_+$, their corresponding $D_c$ are translation and dilation of each other.

\begin{lemma} \label{th:Qs-est} Fix $n \ge 2$. For $x,y \in \R^n_+$, denote $w=x-y$, $z=x-y^*$,  $\th = \frac {x_ny_n}{|z|^{2}}\in (0,\frac14]$, $Q_s$ be as in \eqref{Qs.def}, and
\EQ{
\label{Rs.def}
R_s =  \frac 1{|w|^{s}} - \frac 1{|z|^{s}} ,\quad (s>0);\quad
R_0 =-\log |w|+ \log |z|.
}
For $x \not = y$
we have, for $s \ge 0$,
\EQ{
\label{eq2.38}
0<R_s &\le C_s  |w|^{-s} \th + C_0 1_{s=0} \log(2+\frac {y_n}{|w|}),
\\
 |Q_s | &\le C_s  |w|^{-s} \th^2   + C_0 1_{s=0}  \log(2+\frac {y_n}{|w|}).
}
Above  %
$C_s$ is independent of $x,y\in \R^n_+$.
Moreover, for any $s \ge 0$,
for any  homogeneous polynomial $g(w')$ of degree $\deg g\ge 0$, for any multi-indices $\al$, $\be$ with $\al_n=\be_n=0$ and $m=|\al|+|\be|>0$,
\[
\label{eq2.39}
\nb _x^ {\al} \nb _y ^{\be} \bkt{g(w') R_s )} = \sum_{k=0}^m  f_k(w') R_{s+2k},
\]
\[
\label{eq2.40}
\nb _x^ {\al} \nb _y ^{\be} \bkt{g(w') Q_s } = \sum_{k=0}^m  f_k(w') Q_{s+2k},
\]
for some homogeneous polynomials $f_k(w')$
with $\deg f_k= \deg g +2k-m$.
\end{lemma}

Above and hereafter, the characteristic function $1_\om$ for a condition $\om$ is $1$ if $\om$ is true,
and $0$ if $\om$ is false.
We agree that $f=0$ if it is a polynomial with negative degree. Note that $f_k$ in \eqref{eq2.39} and  \eqref{eq2.40} are the same.

\begin{proof} We first show \eqref{eq2.38}. When $\th>\frac 1{10}$, we have $|w|< c|z|$ for some $c>1$ independent of $x,y$.
Thus \eqref{eq2.38} is trivial if $s>0$, and it is true when $s=0$ because $R_0$ and $Q_0$ are bounded by $1+\log\frac {|z|}{|w|}$, and by using \eqref{zsim1}.

Suppose now $0<\th < \frac 1{10}$. Recall $|w|^2 = |z|^2(1-4\th)$ by \eqref{eq2.29} and $|w|\sim |z|$.
By Taylor expansion,
\[
\label{eq2.42}
 |w|^{-s} = |z|^{-s} (1-4\th)^{-s/2} = |z|^{-s} (1+2s\th + O(\th^2)),
\]
\[
-\log |w| + \log |z| =-\frac 12 \log (1-4\th) = \frac 12(4 \th + O(\th^2)).
\]
Thus \eqref{eq2.38} follows.

Eqn.~\eqref{eq2.39} and \eqref{eq2.40} can be shown by induction on $m$,
using for $j<n$ that
\EQ{
\pd_{x_j} R_s &= -\pd_{y_j} R_s = - d_s R_{s+2}w_j,
\\
\pd_{x_j} Q_s &= -\pd_{y_j} Q_s = - d_s Q_{s+2}w_j,
}
where $d_s=s$ for $s>0$ and $d_0=1$.
\end{proof}

\begin{theorem}
\label{th2.2} Fix $n \ge 2$.  Let $x,y \in \R^n_+$ and $i,j\in\{1,\ldots,n\}$.
Then
\[
\label{th2.2-1}
|G_{ij}(x,y)|\le \frac {C_0 x_ny_n}{|x-y|^{n-2}\cdot|x-y^*|^{2}} + C_0 1_{n=2}
 \log(2+\frac {y_n}{|x-y|}).
\]
Moreover, when $j=n$,
\[
\label{th2.2-2}
|G_{in}(x,y)|\le \frac {C_0 x_n y_n^2}{|x-y|^{n-2}\cdot|x-y^*|^{3}} +  C_0 1_{n=2}  \log(2+\frac {y_n}{|x-y|}).
\]
Above $C_0$ is independent of $x,y\in \R^n_+$.
\end{theorem}

\begin{proof}
Denote $w=x-y$, $z=x-y^*$, and $\th = \frac {x_ny_n}{|z|^{2}}\in [0,\frac14]$.
By Theorem \ref{th2.2a} and Lemma \ref{th:Qs-est}, we have
\[
\label{eq2.51}
|G_{ij}(x,y)| \lec |w|^{2-n}\th^2 +  1_{n=2}
 \log(2+\frac {y_n}{|w|})+
\frac{\th}{|z|^n}  |w_iw_j+z_i\e_j z_j| ,%
\]
which gives \eqref{th2.2-1}.
In the case $j=n$,
\EQ{
w_iw_j+ \e_j z_i  z_j &= w_i (x_n-y_n) - z_i (x_n+y_n)
\\
& = (w_i -z_i)x_n - (w_i +z_i)y_n = -\de_{in} 2y_n x_n - (w_i +z_i)y_n,
}
which is bounded by $|z| y_n $. By this refined estimate and \eqref{eq2.51} we get
\eqref{th2.2-2}.
\end{proof}

{\it Remark.} To prove only \eqref{th2.2-1} without  \eqref{th2.2-2}, it suffices to use $|w|^2 = |z|^2(1+O(\th))$ instead of  \eqref{eq2.42} in the proof of Theorem
 \ref{th2.2a}, and we do not need \eqref{Gij:formula}.
\medskip

We next estimate derivatives of $G_{ij}(x,y)$.

\begin{theorem}
\label{th2.5b}
Fix $n \ge 2$.  Let $x,y \in \R^n_+$ and $i,j\in\{1,\ldots,n\}$. Let $\al$ and $\be$ be multi-indices with
$|\al|+|\be|=m>0$. Then
\[
\label{eq:th24-1}
\abs{\nb_x^\al \nb_y ^\be G_{ij}(x,y)}\le\frac{ C_m }{|x-y|^{n-2+m}}.
\]
If  $\al_n=\be_n=0$, we have
\[
\label{eq:th24-2}
\abs{\nb_x^\al \nb_y ^\be G_{ij}(x,y)}\le  \frac{ C_m x_ny_n}{|x-y|^{n-2+m}|z|^2},
\]
\[
\label{eq:th24-3}
\abs{\nb_x^\al \nb_y ^\be G_{in}(x,y)}\le  \frac{ C_m x_ny_n^2}{|x-y|^{n-2+m}|z|^3}.
\]
If  $\al_n=0$, we have
\[
\label{eq:th24-4}
\abs{\nb_x^\al \nb_y ^\be G_{ij}(x,y)}\le  \frac{ C_m x_n}{|x-y|^{n-2+m}|z|}.
\]
Above $C_m$ are independent of $x,y\in \R^n_+$.
\end{theorem}

\begin{proof}
Estimate \eqref{eq:th24-1} is well-known (see e.g.~\cite[\S IV.3]{Galdi}), and
follows from direct differentiation of
\eqref{Gij:formula}, no matter whether $n>2$ or $n=2$.

Suppose now $\al_n=\be_n=0$.
By \eqref{Gij:formula}, %
\EQ{
\nb_x^\al \nb_y ^\be G_{ij}(x,y)
& =\nb_x^\al \nb_y ^\be \bke{ {\de_{ij}} \ka Q_{n-2} + \frac1{2n\om_n} w_iw_j Q_n}
\\
&\quad
+ x_ny_n
\nb_x^\al \nb_y ^\be \frac{ (w_iw_j+z_i\e_j z_j)}{\om_n |z|^{n+2}}
=:
\text{I}+\text{II}.
}
By Lemma \ref{th:Qs-est},
\EQ{
\text{I}  = \sum_{k=0}^m \bke{ f_{ij,k} (z')  Q_{n-2+2k} + \tilde f_{ij,k}(z') Q_{n+2k}} ,
}
for some homogeneous polynomials $ f_{ij,k}(z') $ and $\td  f_{ij,k}(z') $ with  $\deg f_{ij,k}=2k-m $ and $\deg \td  f_{ij,k} =2+2k-m $. In particular $f_{ij,0}=0$. Thus, by $Q_s$ estimates in  \eqref{eq2.38} with $s>0$,
\[
 \label{eq2.58}
|\text{I}| \lec  \sum_{k=0}^m\bke{  |w|^{2k-m-(n-2+2k)} \th^2 +  |w|^{2+2k-m-(n+2k)} \th^2} \lec  |w|^{2-m-n} \th^2.
\]
For II, using
$w_i=z_i-2\de_{in}y_n$, we may rewrite
\EQ{
w_iw_j+z_i\e_j z_j &=(z_i-2\de_{in}y_n)(z_j-2\de_{jn}y_n)+z_i\e_j z_j
\\
&=(1+\e_j)z_i z_j -2(\de_{jn}z_i+\de_{in}z_j)y_n
+4\de_{jn}\de_{in}y_n^2.
}
Hence
\EQ{
\text{II} =
x_ny_n \nb_x^\al \nb_y ^\be \frac{(1+\e_j)z_i z_j }{\om_n |z|^{n+2}}
-2x_ny_n^2 \nb_x^\al \nb_y ^\be \frac{ \de_{jn}z_i+\de_{in}z_j}{\om_n |z|^{n+2}}
+x_ny_n^3  \nb_x^\al \nb_y ^\be \frac{ 4\de_{jn}\de_{in} }{\om_n |z|^{n+2}} .
}
The factors under differentiation are homogeneous rational functions of $z$ of degrees $-n$, $-n-1$, and $-n-2$, respectively.
After differentiation they become homogeneous rational functions of $z$ of degrees $-n-m$, $-n-1-m$, and $-n-2-m$, respectively. Thus
\EQ{
 \label{eq2.61}
|\text{II}| &\lec  (1+\e_j) x_n y_n   |z|^{-n-m} +x_n y_n ^2  |z|^{-n-1-m}+x_n y_n ^3  |z|^{-n-2-m}.
}

Summing \eqref{eq2.58} and  \eqref{eq2.61} and noting $(1+\e_j)=0$ if $j=n$, we get  both \eqref{eq:th24-2} and \eqref{eq:th24-3}.

It remains to show \eqref{eq:th24-4}. Using above
computations, we note that
\EQ{
\partial_{y_n}^{\be_n}\nb_x^\al \nb_{y'} ^{\be'} G_{ij}(x,y)= & \partial_{y_n}^{\be_n}
\sum_{k=0}^{m-\beta_n} \bke{ f_{ij,k} (z')  Q_{n-2+2k} + \tilde f_{ij,k}(z') Q_{n+2k}}
\\
&+ \partial_{y_n}^{\be_n}\bke{x_ny_n \nb_x^\al \nb_{y'} ^{\be'}
\frac{ (w_iw_j+z_i\e_j z_j)}{\om_n |z|^{n+2}}} =: J_1+J_2.
}
Since $J_2$ is nonsingular and has a factor $x_n$, it is rather
straightforward to obtain \eqref{eq:th24-4}, and thus it suffices to
treat $J_1$ only. In addition, since $f_{ij,k} (z')$ and $\tilde
f_{ij,k}(z')$ are independent of $y_n$-variable, we need to estimate
only $\partial_{y_n}^{\be_n}Q_s$ for either $s=n-2+2k$ or $s=n+2k$.
Recalling that $Q_s=R_s-\frac{2s x_ny_n}{ |z|^{s+2}}$, it is enough
to compute $\partial_{y_n}^{\be_n}R_s$, since the other term can be
treated as $J_2$. We will show via induction argument
that, for $s>0$,
\EQ{
\label{eq2.62}
|\partial_{y_n}^{\beta_n}R_s|\lesssim
\frac{x_n}{\abs{w}^{s+\beta_n}\abs{z}},\qquad \beta_n= 0, 1,\cdots }

The case $\beta_n=0$ follows from \eqref{eq2.38}.
Note
\EQ{
\partial_{y_n}R_s &= s(x_n-y_n)|w|^{-s-2}+
s(x_n+y_n)|z|^{-s-2}
\\
&=s(x_n-y_n)R_{s+2}+2sx_n|z|^{-s-2}.
}
Assume that \eqref{eq2.62} is valid up to
$\beta_n=k\ge 0$, and consider $\beta_n=k+1$:
\EQ{
\partial^{k+1}_{y_n}R_s
& =\partial^{k}_{y_n}\bkt{s(x_n-y_n)R_{s+2}+2sx_n|z|^{-s-2}}
\\
&=s(x_n-y_n)\partial^{k}_{y_n}R_{s+2}-ks\partial^{k-1}_{y_n}R_{s+2}+2sx_n\partial^{k}_{y_n}|z|^{-s-2}.
}
Hence, by induction assumption,
\EQ{
|\partial^{k+1}_{y_n}R_s|
&\lesssim
\frac{x_n(x_n-y_n)}{\abs{w}^{s+2+k}\abs{z}}+\frac{x_n}{\abs{w}^{s+k+1}\abs{z}}+\frac{x_n}{\abs{z}^{s+2+k}}
\lesssim \frac{x_n}{\abs{w}^{s+1+k}\abs{z}}.
}
We can now estimate $J_1$ using that $|f_{ij,k}|\lesssim \abs{z'}^{2k-m}$ and
$|\tilde f_{ij,k}|\lesssim \abs{z'}^{2+2k-m}$,
\[
|J_1|\lesssim\sum_{k=0}^{m-\beta_n}\bke{
\frac{\abs{z'}^{2k-m}x_n}{\abs{w}^{n-2+2k+\beta_n}\abs{z}}+\frac{\abs{z'}^{2+2k-m}x_n}{\abs{w}^{n+2k+\beta_n}\abs{z}}}
\lesssim \frac{x_n}{\abs{w}^{n-2+m+\beta_n}\abs{z}}.
\]
This completes the
proof.
\end{proof}

\begin{remark} When $\al_n\not=0$ or $\be_n\not=0$, we do not expect
\eqref{eq:th24-2} since $\pd_{x_n}$ or $\pd_{y_n}$ may kill a factor
of $x_n$ or $y_n$.  For example, consider the Green function for the
Laplace equation in $\R^3_+$,
\[
G(x,y) = \frac 1{4\pi|x-y|} - \frac 1{4\pi|x-y^*|}.
\]
We have
\[
4\pi\pd_{x_3}G(x,y)= x_3\bke{-\frac 1{|x-y|^3} + \frac 1{|x-y^*|^3}} +
y_3\bke{\frac 1{|x-y|^3} + \frac 1{|x-y^*|^3}}.
\]
When $y=e_3$ and $1\le x_3 \ll |x|$, the first term on the right side
is of order $\frac {x_3^2}{|x|^5}$ but the second term is of order
$\frac {1}{|x|^3}$. Thus $|\pd_{x_3}G(x,e_3)|\not \lesssim \frac
{x_3}{|x|^4}$.  However, \eqref{eq:th24-2} may be still valid if
$k<n$:
\[
4\pi\pd_{x_1}G(x,y)= (x_1-y_1)\bke{-\frac 1{|x-y|^3} + \frac
  1{|x-y^*|^3}},
\]
which is $O\bke{\frac {x_3}{|x|^4}}$ if $y_3=1\le x_3 $.
\end{remark}

The following lemma will be used in the proof of Theorem \ref{stokes-100}.

\begin{lemma}
\label{th2.7}
For $n \ge 2$, we have
\begin{equation}\label{DGy0}
D_{y_l}G_{ij}(x,0) = \left\{
\begin{split}
& \frac{2 x_n x_i
x_j}{ \omega_n |x|^{n+2}} = K_{ij}(x),\qquad &\mbox{if }\,\, j<n=l,
\\
 & 0,\qquad &\text{otherwise}.
\end{split}
\right .
\end{equation}
\end{lemma}
\begin{proof}
Recall formula \eqref{Gij:formula} for $G_{ij}(x,y)$ in Theorem \ref{th2.2a}. Note that $Q_s|_{y_n=0}=0$ for all $s \ge 0$. By \eqref{eq2.40} of Lemma \ref{th:Qs-est}, we get 
\[
\nabla^\beta_y Q_s |_{y_n=0} = 0, \quad \forall s\ge0 , \forall \beta.
\]
Therefore, when one computes $D_{y_l}G_{ij}(x,0)$ using \eqref{Gij:formula}, the first two terms have no contribution and
\[
D_{y_l}G_{ij}(x,0) =0+0+ \frac {x_n  \de_{ln} (1+\e_j)x_ix_j}{\om_n |x|^{n+2}},
\]
which shows the lemma.
\end{proof}
\cmtm{\emph{Remark.} It seems interesting to show the lemma by definitions, not using formula \eqref{Gij:formula}. Compare the derivation of the Poisson kernel for the Laplace equation from its Green function.}
\section{Asymptotics of flows in the half space} \label{sec3}

In this section, we study the spatial asymptotics of stationary
solutions of the incompressible Stokes and Navier-Stokes equations in the half
space $\R^n_+$. %

We first consider the Stokes system in the half-space,%
\begin{equation}\label{stokes-KMT-10}
-\Delta v+\nabla p=f+\nabla\cdot F, \qquad{\rm div}\,v=0\qquad
\mbox{in }\,\,\R^n_+,
\end{equation}
\begin{equation}\label{stokes-KMT-20}
v=0\qquad \mbox{on }\,\,\partial\R^n_+=\{x_n=0\}.
\end{equation}
Above $(\nabla\cdot F)_i = \pd_j F_{ji}$. \cmtm{A \emph{weak solution} of \eqref{stokes-KMT-10}-\eqref{stokes-KMT-20} is a vector field $v \in  W^{1,2}_{loc}(\overline {\R^n_+})$ that satisfied the weak form of \eqref{stokes-KMT-10} with divergence-free test functions, and \eqref{stokes-KMT-20} in the sense of trace, with no assumption on its global integrability in this section.

The following uniqueness result can be found in e.g.~\cite[Corollary 3.7]{MR2027755}.
\begin{lemma}[Uniqueness in $\R^n_+$] \label{th:uniqueness} 
Let $v \in W^{1,2}_{loc}(\overline {\R^n_+})$, $n \ge 2$, be a weak solution of the Stokes system \eqref{stokes-KMT-10}-\eqref{stokes-KMT-20} with zero force. If $v(x)=o(|x|)$ as $|x|\to \infty$, then $v \equiv 0$.
\end{lemma}
}

The following two lemmas show that we can absorb $f $ into $\nb \cdot F$.

\begin{lemma}[\cite{KMT12} Lemma 2.5]
\label{th:f.dec}
   If $f(x)$ is defined in $\R^n$ with $|f(x)| \lec
\bka{x}^{-a}$, $a>n\ge 1$, then for any $R>0$ we can rewrite
\begin{equation}
\label{f.dec} f(x) = f_0(x) + \sum_{j=1}^n \pd_j F_j(x)
\end{equation}
where $\supp f_0\subset B_R(0)$ and $|F_j(x)| \lec
\bka{x}^{-a+1}\norm{\bka{x}^a f(x)}_{L^\I}$.
\end{lemma}

If we are concerned with the half space, 
the term $f_0$ can be removed.

\begin{lemma}\label{th:f.dec+}
   If $g(x)$ is defined in $\R^n_+$ with $|g(x)| \lec
\bka{x}^{-a}$, $a>n\ge 1$, then we can rewrite
\begin{equation}
\label{g.dec} g(x) = \sum_{j=1}^n \pd_j G_j(x),\quad (x\in\R^n_+),
\end{equation}
where $|G_j(x)| \lec \bka{x}^{-a+1}\norm{\bka{x}^a g(x)}_{L^\I}$.
\end{lemma}
\begin{proof}
Let
\begin{equation}
f(x) = \left\{
\begin{split}
&g(x-e_n),\quad (x_n>1);
\\
&0,\quad (x_n\le 1).
\end{split}
\right.
\end{equation}
By Lemma \ref{th:f.dec}, we can decompose $f(x)$ as in \eqref{f.dec}
with $\supp f_0\in B_{1/2}(0)$. Let $G_j(x)=F_j(x+e_n)$ and we get
\eqref{g.dec} with the desired decay estimate.
\end{proof}

If the external forces $f$ and $F$ decay sufficiently fast, then
bounded solutions of \eqref{stokes-KMT-10}-\eqref{stokes-KMT-20}
have spatial asymptotics of order $-n+1$. To be more precise, we
have the following:

\begin{theorem}[Asymptotics of Stokes system] \label{stokes-100}
Let $n \ge  2$ and $a>n+1$. Suppose that $v$ is a weak
solution of the Stokes system
\eqref{stokes-KMT-10}-\eqref{stokes-KMT-20}  in $\R^n_+$ with
$\abs{v(x)}\le o(\abs{x})$ as $|x|\to \I$. Assume further that
$\abs{f(x)}\lesssim \langle x \rangle^{-a}$ and $\abs{F(x)}\lesssim
\langle x \rangle^{-a+1}$. Then, $\abs{v(x)}\lesssim \frac {x_n}{
\bka{x}^n} \lec \langle x \rangle^{-n+1}$ and for sufficiently large
$x$,
\[
\label{eq3.6}
v_i(x)=\sum_{j=1}^{n}K_{ij}(x)b_j+O(\de(x)),
\quad i=1,\cdots,n,
\]
where
\[
\label{eq3.12}
b_n=0,\quad b_j=\int_{\R^n_+}(y_n f_j(y) - F_{nj}(y))dy, \quad
(j<n),
\]
\[
\de(x) =\frac {x_n}{\bka{x}^{\min(n+1,a-1)}} (1+1_{a=n+2} \log \bka{x}) ,
\]
and $K_{ij}(x)=\frac{2 x_n x_i x_j}{ \omega_n |x|^{n+2}}$ is the Poisson kernel
for the Stokes system in the half space.
\end{theorem}

{\it Remark.} 
Note $\de(x) = o( \frac {x_n}{\bka{x}^n})$ as $|x| \to \infty$.
The asymptotic in \eqref{eq3.6} is spanned by the $n-1$ vectors $\{\vec K_j : 1 \le j \le n-1\}$, 
where $(\vec K_j)_i = K_{ij}$. That $\vec K_n$ is not present is because a solution of \eqref{stokes-KMT-10}-\eqref{stokes-KMT-20} should have zero flux on any hemisphere $S_R^+= \bket{ x\in \R^n_+, |x|=R}$, while $\vec K_n$  has nonzero flux. Note that the flux of the error term of \eqref{eq3.6} on $S_R^+$ vanishes as $R \to \I$.

\begin{proof}
By Lemma \ref{th:f.dec+}, we may write $f_j  =\pd_i \td F_{i j}$ where $|\td F_{i j}(x)|\lec \bka{x}^{-a-1} $. We have
\[
\int_{\R^n_+} y_n f_j(y) dy = \int_{\R^n_+} y_n \pd_i \td F_{i j}(y) dy= - \int_{\R^n_+} \td F_{n j}.
\]
By absorbing $\td F$ into $F$, we may assume $f=0$.

By uniqueness Lemma \ref{th:uniqueness}, we have
the representation formula,
\EQ{\label{eq3.10}
v_i(x) &=-\sum_{j=1}^n\sum_{\alpha=1}^n\int_{\R^n_+}\partial_{y_{\alpha}}G_{ij}(x,y)F_{\alpha
j}(y)dy =I_1+I_2,
}
where
\EQ{
I_1&:=
-\sum_{j=1}^n\sum_{\alpha=1}^n\int_{\R^n_+}\partial_{y_{\alpha}}G_{ij}(x,0)F_{\alpha
j}(y)dy,
\\
I_2&:= -\sum_{j=1}^n\sum_{\alpha=1}^n\int_{\R^n_+}\bke{\partial_{y_{\alpha}}G_{ij}(x,y)
-\partial_{y_{\alpha}}G_{ij}(x,0)}F_{\alpha j}(y)dy.
}

We first compute $I_1$. By \eqref{DGy0} in Lemma \ref{th2.7}, the summand in $I_1$ is nonzero only if $j<n=\alpha$ and
\EQ{
I_1 &
=- \sum_{j=1}^{n-1} \frac{2 x_n x_i x_j}{\omega_n |x|^{n+2}}\int_{\R^n_+}
F_{n j}(y)dy
= \sum_{j=1}^{n-1}  K_{ij}(x) b_j.
}

Secondly, we estimate $I_2$. We may assume $|x|>10$. For notational convenience, for
given $x$ we denote $A_x=\{y\in\R^n_+: \abs{y}\le
\frac{\abs{x}}{2}\}$ and $B_x=\R^n_+\setminus A_x$.
\EQ{
\label{eq3.13}
I_2 &=-\sum_{j=1}^n\sum_{\alpha=1}^n\int_{\R^n_+}\bke{\partial_{y_{\alpha}}G_{ij}(x,y)
-\partial_{y_{\alpha}}G_{ij}(x,0)}F_{\alpha j}(y)dy
\\
&=\int_{A_x}\cdots dy+\int_{B_x}\cdots dy:=J_1+J_2.
}
By \eqref{eq:th24-4} of Theorem \ref{th2.5b},
\EQ{
|J_1| &\le C\int_{\abs{y}\le \frac{\abs{x}}{2}}
 \sup_{\abs{\td y}\le \frac{\abs{x}}{2}}
\abs{\nabla^2_ y G_{ij}(x,\td y)
}\abs{yF_{\alpha j}(y)}dy
\\
&\le \frac{C x_n}{\abs{x}^{n+1}}\int_{\abs{y}\le
\frac{\abs{x}}{2}} %
\bka{y}^{2-a} dy 
\le C\de(x).
}
Above we have used that, for $m\ge 0$ and $R>0$,
\begin{equation}
\label{eq3.15}
\int_{|y|\le R} \bka{y}^{-m}dy \lec
 \left \{
\begin{split}
1 \quad & \text{if } m>n,
\\
\log \bka{R}\quad & \text{if } m=n,
\\
\bka{R}^{n-m} \quad & \text{if } 0\le m<n,
\end{split}
\right \}
\approx
 \bka{R}^{(n-m)_+} (1+1_{m=n} \log \bka{R}),
\end{equation}
with $m=a-2$. Recall $(r)_+= \max(r,0)$.
For $J_2$,
\EQ{
|J_2| &\le
\sum_{j=1}^n\sum_{\alpha=1}^n\int_{\frac{\abs{x}}{2}<\abs{y}}\bke{\abs{\partial_{y_{\alpha}}G_{ij}(x,y)}+
\abs{\partial_{y_{\alpha}}G_{ij}(x,0)}}\abs{F_{\alpha j}(y)}dy
\\
& \leq C\int_{\frac{\abs{x}}{2}<\abs{y}}
\bke{ \frac {x_n}{|x-y|^{n-1}|x-y^*|}  + \frac {x_n}{|x|^n}} |y|^{1-a} dy
\le C\frac{x_n} {\abs{x}^{a-1}}.
}
This completes the proof.
\end{proof}

\cmtm{
\begin{remark} \label{rem3.4}
If $a\le n+1$, the integral \eqref{eq3.12} for $b_j$ diverges and the asymptotic formula \eqref{eq3.6} is meaningless. However, the integral \eqref{eq3.10} still converges if $1<a<\infty$ and %
\EQ{
|v(x)| &\lec
 \int_{\R^n_+}  \frac {x_n}{|x-y|^{n-1}|x-y^*|}  \bka{y}^{1-a} dy.
}
By estimating the integral in the two regions $\bket{|y|<|x|/2}$ and
$\bket{|y|>|x|/2}$ separately as in \eqref{eq3.13}, 
\EQ{
|v(x)| &\lec \frac{ x_n} {\bka{x}^{\min(n,a-1)}}\, (1+ 1_{a=n+1} \log \bka{x}).
}

\end{remark}
}

Next we consider the Navier-Stokes equations in the half-space, i.e.,
\begin{equation}\label{NSe-KMT-30}
-\Delta u+(u\cdot\nabla )u+\nabla p=f+\nabla\cdot F, \qquad{\rm
div}\,u=0\qquad \mbox{in }\,\,\R^n_+,
\end{equation}
\begin{equation}\label{NSe-KMT-40}
u=0\qquad \mbox{on }\,\,\partial\R^n_+=\{x_n=0\}.
\end{equation}
\cmtm{A \emph{weak solution} $u$ of \eqref{NSe-KMT-30}-\eqref{NSe-KMT-40} is a weak solution of 
 \eqref{stokes-KMT-10}-\eqref{stokes-KMT-20} with force $f + \nabla \cdot (F - u \otimes u)$.}
 
If the decay rates of external forces $f$ and $F$ are sufficiently
fast with small coefficient, there exist solutions of the
Navier-Stokes equations \eqref{NSe-KMT-30}-\eqref{NSe-KMT-40}, whose
spatial asymptotics is of $-n+1$-order. Our result reads as follows:

\begin{theorem} [Existence and aysmptotics of NSE] \label{nse-200}
Let $n \ge 3$ and $a>n+1$. There exists $\epsilon_0>0$
such that if $\abs{f(x)}\le \epsilon\langle x \rangle^{-a}$ and
$\abs{F(x)}\le \epsilon\langle x \rangle^{-a+1}$ with $\epsilon<\epsilon_0$, then there
exists a weak solution $u$ of the Navier-Stokes equations
\eqref{NSe-KMT-30}-\eqref{NSe-KMT-40} in $\R^n_+$ with $\abs{u(x)}\lesssim  \frac{ \epsilon x_n}{\bka{x}^n} \lec
\epsilon \langle x \rangle^{-n+1}$ and, furthermore, its asymptotics
is given as
\begin{equation}\label{asym-nse-100}
u_i(x)=\sum_{j=1}^{n}K_{ij}(x)\tilde{b}_j+O(\cmtm{\e} \td \de(x)),
\end{equation}
where
\begin{equation}\label{asymptotics-KMT-200}
\tilde{b}_n=0,\quad \tilde{b}_j = \int_{\R^n_+} \bket{ u_n(y)u_j(y)
+y_n f_j(y)-F_{nj}(y)} dy, \quad (j <n),
\end{equation}
\[
\label{asymptotics-KMT-201}
\td \de(x) =\frac {x_n}{\bka{x}^{\min(n+1, a-1)}}
 (1+1_{\td a=n+2} \log \bka{x}), \quad \td a = \min (a, 2n-1),
\]
and $K_{ij}(x)=\frac{2 x_n x_i x_j}{ \omega_n |x|^{n+2}}$ is the Poisson kernel for the Stokes system in the half
space.
\end{theorem}

Unlike Theorem \ref{stokes-100}, the case $n=2$ is not included in Theorem \ref{nse-200}. Note 
$2n-1\ge n+2$ and
$\tilde a>n+1$ due to $n>2$.

\begin{proof}
As in the proof of Theorem
\ref{stokes-100}, we may assume $f=0$.
Let
\begin{equation}\label{Jan4-KMT-10}
\calK=\bket{ v\in C(\overline{ \R^n_+}; \R^n) : v|_{\pd \R^n_+}=0, \, \norm{v}_{\calK}:=\sup_{x\in\R^n_+}{\langle
x\rangle}^{n-1} \abs{v(x)}< C\epsilon},
\end{equation}
where $0<\epsilon\ll 1$ is sufficiently small and will be
specified later. %
Now we set $v^1(x)=0$ and
iteratively define $v^{k+1}$, $k=1,2,\cdots$, by
\[
v^{k+1}_i(x)  = \int _{\R^n_+} (\pd_{y_\al} G_{ij}) (x,y) (v^k_\al v^k_j -F_{\al j})(y)\,dy,
\]
which solves the Stokes system
\[
-\Delta v^{k+1}+\nabla p^{k+1}=-(v^k\cdot\nabla )v^k+\nabla\cdot F,
\qquad{\rm div}\,v^{k+1}=0\qquad \mbox{in }\,\,\R^n_+,
\]
\[
v^{k+1}=0\qquad \mbox{on }\,\,\partial\R^n_+.
\]
Due to Theorem \ref{stokes-100}, we have $\norm{v^{k+1}}_{\calK}\leq
C\epsilon$ uniformly for all $k=1,2,\cdots$, if $\epsilon$ is sufficiently small.
If we set $\delta v^k:=v^{k+1}-v^{k}$, we have
\[
\de v^{k+1} _i(x)  = \int _{\R^n_+} (\pd_{y_\al} G_{ij}) (x,y) (v^{k+1}_\al\de v^k_j + v^k_j  \de v^k_\al  )(y)\,dy,
\]
and hence
\[
\norm{\de v^{k+1}}_{\calK}\leq
C\epsilon\norm{\de v^k }_{\calK}.
\]
The argument of contraction mapping gives a unique solution $u$ of
\[
u_i(x)  = \int _{\R^n_+} (\pd_{y_\al} G_{ij}) (x,y) (u_\al u_j -F_{\al j})(y)\,dy,
\quad \norm{u}_{\calK} \le C \epsilon.
\]

Finally, we may consider $u$ as a solution of the Stokes system with force tensor $F_{\al j}-u_\al u_j$.
Since $ |(F_{\al j}-u_\al u_j )(x)| \le C \epsilon \bka{x}^{-\td a+1}$ with $\td a= \min (a, 2n-1)$,
Theorem \ref{stokes-100} gives the desired asymptotics
\eqref{asym-nse-100}--\eqref{asymptotics-KMT-201}.
\end{proof}

Next theorem shows that for any vector $b=(b_1,\cdots, b_{n-1}, 0)$
with small magnitude the Navier-Stokes equations
\eqref{NSe-KMT-30}-\eqref{NSe-KMT-40} has a solution whose leading
asymptotics is $\sum_{j=1}^{n-1}K_{ij}b_j$. More precisely, we have
the following.

\begin{theorem}[Asymptotic completeness] \label{nse-600}
Let $n\ge 3$. There exists a small number $\epsilon_1>0$ such
that if  $b=(b_1,\cdots, b_{n-1},0)$,
 $\abs{b}=\epsilon<\epsilon_1$, then there exists a smooth 2-tensor $F$ supported in $B_1 \cap \R^n_+$ and a
weak solution $u$ of the Navier-Stokes equations
\eqref{NSe-KMT-30}-\eqref{NSe-KMT-40} corresponding to this $F$ and zero $f$, satisfying
\begin{equation}\label{asym-Dec29-10}
u_i(x)=\sum_{j=1}^{n-1}K_{ij}b_j+ \cmtm{ O\bke{\td \de(x) },}
\end{equation}
\cmtm{ where $\td \de(x)$ is given by \eqref{asymptotics-KMT-201}.}
\end{theorem}

\begin{proof}
Fix any smooth scalar function $\phi$ supported in $B_1 \cap \R^n_+$ with $\int \phi = 1$.
For small $a=(a_1, \cdots, a_{n-1})$, define 2-tensor $F^a$ by
\[
F^a_{ij} = 0 \quad \text{if} \quad i<n; \quad F^a_{ij} = -a_j \phi  \quad \text{if} \quad i=n.
\]
By Theorem \ref{nse-200}, there is a solution $u^a$ of the Navier-Stokes equations with force $F^a$ and zero $f$ if $|a|\le \epsilon_0$ for some small $\epsilon_0>0$. We have $|u^a(x)| \le C |a| \bka{x}^{1-n}$. The coefficients $(\td b_1, \cdots, \td b_{n-1})$ of the leading term in \eqref{asym-nse-100} for $u^a$ will be denoted as $B_{NS}(a)$. Thus
$B_{NS}(a)_j = a_j +  \int _{\R^n_+} u^a_ n u^a_j $, and for some $C_1$,
\[
|B_{NS}(a)_j- a_j| =  |\int _{\R^n_+} u^a_ n u^a_j | \le \frac{C_1}n |a|^2, \quad |B_{NS}(a)- B_{NS}(\td a)| \le C_1 |a-\td a|.
\]

For given small $b$ we want to solve $a$ so that $B_{NS}(a)=b$. This equation can be rewritten as a fixed point problem
\[
a = \Phi(a), \quad  \Phi(a) : = a - B_{NS}(a) + b.
\]
Denote $D_r = \bket{a \in \R^{n-1}: |a| \le r}$. One checks easily that $\Phi$ is continuous on $D_{\epsilon_0}$. Denote $\epsilon_1 =
\min(\epsilon_0/2, \frac 1{4C_1})$. Suppose $|b|=\epsilon \le \epsilon_1$. For
$a \in D_{2\epsilon}$, we have
\[
|\Phi(a)| \le |a - B_{NS}(a) |+ |b| \le C_1(2\epsilon)^2 + \epsilon \le 2  \epsilon.
\]
Thus $\Phi$ is a continuous map that maps the closed disk $D_{2\epsilon}$ into itself. By Brouwer fixed point theorem, $\Phi$ has a fixed point
in $D_{2\epsilon}$. This completes the proof.
\end{proof}

The next theorem is an application of Theorem \ref{nse-200} and considers the asymptotics of any \emph{given} solution.

\begin{theorem}[Asymptotics] \label{application-KMT-100}
Let $n \ge 3$ and $a> n+1$. Suppose that $u\in W^{1,2}_{loc}(\overline{\R^n_+} )$ is a weak solution of the Navier-Stokes
equations \eqref{NSe-KMT-30}-\eqref{NSe-KMT-40} with force $f+ \nb\cdot F$.

\begin{itemize}
\item[(i)] Suppose $|u(x)| \le C \bka{x}^{-m}$, $m > \max\bket{\frac{n-2}2, \frac  {n-1}3, \frac n4}$,  $\abs{f(x)}\le \epsilon\langle x
\rangle^{-a}$ and $\abs{F(x)}\le \epsilon\langle x
\rangle^{-a+1}$ for sufficiently small $\epsilon$.
Then, $u$ agrees with the solution of Theorem \ref{nse-200},  $\abs{u(x)} \le C\epsilon x_n\langle x \rangle^{-n}$,
 and its asymptotics is
given by \eqref{asym-nse-100} with $\tilde b_j$ and $\tilde \delta(x)$ given by \eqref{asymptotics-KMT-200} and 
\eqref{asymptotics-KMT-201}.

\item[(ii)]
Suppose $\abs{f(x)}\le C \langle x \rangle^{-a}$ and
$\abs{F(x)}\le C \langle x \rangle^{-a+1}$, and
$\abs{u(x)}\le C\langle x \rangle^{-1-\sigma}$ for some $\sigma>0$.
Then, $\abs{u(x)} \le Cx_n\langle x \rangle^{-n}$,  and its asymptotics is given as
\begin{equation}\label{asym-nse-100b}
u_i(x)=\sum_{j=1}^{n}K_{ij}(x)\tilde{b}_j+O( \td \de(x)),
\end{equation}
with $\tilde b_j$ and $\tilde \delta(x)$ given by \eqref{asymptotics-KMT-200} and 
\eqref{asymptotics-KMT-201}.

\item[(iii)]
Suppose $\abs{f(x)}\le C \langle x \rangle^{-a}$ and
$\abs{F(x)}\le C \langle x \rangle^{-a+1}$, and $|u(x)| \le \epsilon \bka{x}^{-1}$
 for sufficiently small $\epsilon$. Then,
$\abs{u(x)} \le C x_n\langle x \rangle^{-n}$, and its asymptotics is
given by \eqref{asym-nse-100b} with $\tilde b_j$ and $\tilde \delta(x)$ given by \eqref{asymptotics-KMT-200} and 
\eqref{asymptotics-KMT-201}.
\end{itemize}

\end{theorem}

Note that Case (i) assumes small $f$ and $F$ but allows large $u$,  Case (ii) allows large $f$, $F$ and $u$ but assumes extra decay, and Case (iii)
assumes small $u$ but allows large $f$ and $F$. \cmtm{Also note that we do not claim smallness in Case (iii). The error estimate has a small factor $\e$ only in Case (i).}

\begin{proof}
As in the previous proofs, we assume that $f=0$ without loss
of generality.

\medskip

\noindent
$\bullet$ Case (i).\quad We may assume $m<n-1$.
By Theorem \ref{nse-200}, there exists a solution
$\tilde{u}$, which satisfies the conclusion of Theorem
\ref{nse-200}. Thus, it suffices to show $u=\tilde{u}$. Set $w=u-\tilde{u}$ and $q=p-\tilde{p}$.
We get
\[
\label{eq3.36}
-\Delta w+\nabla q=-(u\cdot\nabla )w- (w\cdot\nabla )\tilde{u},
\qquad{\rm div}\,w=0\qquad \mbox{in }\,\,\R^n_+,
\]
\[
w=0\qquad \mbox{on }\,\,\partial\R^n_+=\{x_n=0\}.
\]
{By \cite[Theorem 3.4]{MR2027755}, for $R>1$,
\[
\label{eq3.39}
\norm{\nb w}_{L^2(B_R^+)} + \norm{q - q_R}_{L^2(B_R^+)}
\le C 
\norm{ u \otimes w+ w \otimes \td u}_{L^2(B_{2R}^+)} + \frac CR \norm{ w}_{L^2(B_{2R}^+)}\le C,
\]
where $q_R = |B_{2R}^+ |^{-1} \int_{B_{2R}^+} q$ and $C$ is independent of $w$ and $R$. The second inequality is due to $|u(x)|+|\td u(x)|+|w(x)| \lec \bka{x}^{-m}$. 
In particular $\nb w \in L^2(\R^n_+)$.

Let $ Z \in C^2(\R)$ with $0 \le Z(t) \le 1$, $Z(t)=0$ for $t>1$ and $Z(t)=1$ for $t<1/2$. Let $\zeta= Z(|x|/R)$. 
Testing \eqref{eq3.36} with  $w\zeta^2$ and integrating by parts, we get
\EQ{
\int |\nb (w\zeta)|^2 & = 
\int \bket{ qw_i \pd_i \zeta^2 + \frac { |w|^2}2 u_i \pd_i \zeta^2
+ |w|^2 |\nb \zeta|^2
+ w_j \tilde u _i \bkt{w_i \zeta \pd_j \zeta  +\zeta \pd_j(w_i \zeta) }}
= \sum_{j=1}^5 I_j.
}
Note $|I_2| + |I_4| \lec R^{n-1-3m}$, $|I_3| \lec R^{n-2-2m}$, 
Also, 
\[
I_1 = \int (q-q_R)w \cdot \nb \zeta ^2\le \norm{q-q_R}_{L^2(B_R^+)} \norm{w \cdot\nb \zeta^2}_{L^2(B_R^+)} \lec R^{\frac n2 -1-m}
\]
using \eqref{eq3.39}.
Finally,
\[
|I_5|\le \norm{\td u}_{L^n}  \norm{w\zeta}_{L^{\frac{2n}{n-2}}} \norm{\nb (w\zeta)}_{L^2} \le C\e \int |\nb (w\zeta)|^2 .
\]
If $C \e <1$, we get $\int |\nb (w\zeta)|^2 \le o(1)$.
Taking $R \to \infty$, we get $\nb w=0$ and $w=0$.
}

\medskip
\noindent
$\bullet$ Case (ii).\quad
We may assume $0<\sigma<n-2$. By uniqueness (Lemma \ref{th:uniqueness}), we have the representation formula, 
\[
u_i(x)=  \sum_{j=1}^n \sum_{\alpha=1}^n \int_{\R^n_+} \partial_{y_{\alpha}} G_{ij}(x,y)\bke{u_{\alpha} u_j - F_{\alpha
j}}(y)\,dy.
\]
The contribution from $F$ is bounded by $\frac {Cx_n}{\bka{x}^n}$ by Theorem \ref{stokes-100}. 
The nonlinearity satisfies $|u_\al u_j(y)| \le C \bka{y}^{1-a'}$ with $a '= 2\sigma+3>1$. By Remark \ref{rem3.4}, 
\EQ{
|u(x)| 
&\lec  \frac{ x_n} {\bka{x}^{\min(n,a'-1)}}\, (1+ 1_{a'=n+1} \log \bka{x})+\frac {x_n}{\bka{x}^n}.
}
If $a'-1\le n$, we can avoid  the log factor by taking a slightly smaller $a'$ and
we get 
\[
|u(x)| \lec  \frac { x_n}  {\bka{x}^{ \frac 32 \sigma+2}} \lec \bka{x}^{-1-\frac 32 \sigma}. 
\]
We can repeat this procedure until we obtain $a'-1>n$ and hence
 $\abs{u(x)}\leq
\frac{Cx_n}{\abs{x}^{n}}$. We then use Theorem \ref{stokes-100} to get its asymptotics. %

\medskip
\noindent
$\bullet$ Case (iii).\quad Fix $\sigma \in (0,1)$.
We will construct a solution $v$ satisfying
$\abs{v(x)}\le C\langle x \rangle^{-1-\sigma}$ and the following perturbed equations
\[
\label{eq3.45}
-\Delta v+\nabla \pi+(u\cdot\nabla)v=\nabla F,\quad {\rm div}\,
v=0\qquad \mbox{in }\,\,\R^n_+,
\]
\[
\label{eq3.46}
v=0\qquad \mbox{on }\,\,\partial\R^n_+=\{x_n=0\}.
\]
This can be done by iteration: Let $v^{(0)}=0$ and define $v^{k+1}$ for $k \ge 0$ by
\[
v^{k+1}_i(x)  = \int _{\R^n_+} (\pd_{y_\al} G_{ij}) (x,y) (u^k_\al v^k_j -F_{\al j})(y)\,dy.
\]
For $\e$ sufficiently small, 
 $\abs{v^{(k+1)}(x)}\le C\langle x
\rangle^{-1-\sigma}$ uniformly in $k$ using Remark \ref{rem3.4}, and 
converges to some $v$ with the same bound. The difference $w=u-v$ satisfies 
\[
w_i(x)  = \int _{\R^n_+} (\pd_{y_\al} G_{ij}) (x,y) (u^k_\al w^k_j )(y)\,dy, \quad |w(x)| \le C \bka{x}^{-1-\sigma}.
\]
By Remark \ref{rem3.4}, 
\[
\norm{ \bka{x}^{1+\si} w(x)}_{L^\infty} \le C \e \norm{ \bka{x}^{1+\si} w(x)}_{L^\infty}.
\]
Thus, if $\e$ is sufficiently small, $w=0$ and $|u(x)| \le C \bka{x}^{-1-\sigma}$.

By Case (ii), we deduce $|u(x)| \le C x_n\bka{x}^{-n}$. %
\end{proof}

Another application of Theorem \ref{nse-200} is on asymptotic profiles
of solutions for the aperture type problem of the Navier-Stokes
equations.
Let $\Sigma_r:=\partial\R^n_+ \setminus B_r=\{(x',0): \abs{x'}\ge
r\}$ and $\Omega_r=\R^n_+\setminus \overline B_r$, and consider
\begin{equation}\label{aperture-100}
-\Delta u+(u\cdot\nabla )u+\nabla p=0, \qquad{\rm div}\,u=0\qquad
\mbox{in }\,\,\Omega_r,
\end{equation}
\begin{equation}\label{aperture-200}
u=0\qquad \mbox{on }\,\,\Sigma_r .
\end{equation}
We emphasize that no boundary condition is imposed on $\partial
B_r\cap\R^n_+$. 

Suppose that $|u(x)| \lec |x|^{-1-\de}$ for large $x$ and is small for $r\le |x|\le \rho<\infty$. Choose
 $r<l_1<l_2<\rho$ and let $\zeta$
be a smooth cut-off function satisfying 
\EQ{
\zeta(x)= \left\{
\begin{array}{cc}
1&\mbox{ if }  \abs{x}>l_2\\
0&\mbox{ if }  \abs{x}<l_1.
\end{array}
\right.
}
 Recall that $\vec{K}_n=(K_{1n}, \cdots, K_{nn})$ is the Poisson kernel for the Stokes system in the half space with $j=n$ and $k_n$ is the pressure corresponding
to $\vec{K}_n$.
Set 
\EQ{ \label{eq3.51}
w=(u-\tilde{b}_n K_n)\zeta+\tilde{w}, \quad \pi=(p-\tilde{b}_n k_n)\cmtm{\zeta},
}
where $\tilde{w}$
solves
\[
{\rm{div}}\,\tilde{w}=(u-\tilde{b}_n K_n)\cdot\nabla\zeta\quad\mbox{
in }\,\,B_{l_2}^+\setminus B_{l_1}, \qquad\qquad
\tilde{w}=0\quad\mbox{ on }\,\,\partial (B_{l_2}^+\setminus B_{l_1}).
\]
Then, $w$ solves
\[
\label{eq3.51+}
-\Delta w+(w\cdot\nabla )w+\nabla \pi=f+\nabla\cdot F, \qquad{\rm
div}\,w=0\qquad \mbox{in }\,\,\R^n_+,
\]
\[
\label{eq3.52}
w=0\qquad \mbox{on }\,\,\partial\R^n_+=\{x_n=0\},
\]
where $f$ and $F=(F_{\alpha j})_{\al,j=1,\ldots,n}$ are given by
\[
\label{KMT-aperture-100}
f= \cmtm{ (u-\tilde{b}_n K_n)\Delta\zeta}
+(p-\cmtm{\tilde{b}_n} k_n)\nabla\zeta+u\cdot\nabla\zeta u,
\]
\EQ{
\label{KMT-aperture-200}
F_{\alpha j}=- \cmtm{2}{\partial_{\alpha}\zeta(u-\tilde{b}_n K_n)_j}  
-\delta_{\alpha j}\partial_{j}\tilde{w}
+ \tilde{w}_{\alpha}((u-\tilde{b}_n
K_n)\zeta+\tilde{w})_j
+{(u-\tilde{b}_n K_n)_\alpha\zeta \tilde{w}_j}
\\
-{(\tilde{b}_n
K_n)_\alpha\zeta (u-\tilde{b}_n K_n)_j\zeta
}
-{u_{\alpha}\zeta(\tilde{b}_n K_n)_j\zeta }
-{u_{\alpha}(1-\zeta)u_j\zeta }.
}
Note that $f$ is small and has compact support, while $F$ is small and decays like $|x|^{1-a}$, $a= n+1+\de$. Thus, if $\de>0$,
the asymptotic profile of $u$ is  $\td b_n K_n$ plus that of $w$  given by \eqref{asymptotics-KMT-200}. 
To be more precise, we have the
following.

\begin{theorem}  [{Aperture type problem}] \label{nse-500}
Let $n \ge 3$, $0<r<\rho<\infty$, and $0<\de<1$. There is a small $\e_0>0$ such that, if $u 
\in  H^1_{loc}(\overline\Omega_r)$ is a weak 
solution of \eqref{aperture-100} and \eqref{aperture-200} in $\Om_r$ satisfying
$\abs{u(x)}\le \langle x
\rangle^{-1-\delta}$ and $\e=\norm{u}_{L^\infty(B_\rho \setminus B_r)} \le \e_0$, then $\abs{u(x)}\lesssim \ve \langle x
\rangle^{-n+1}$ in $\Om_r$ and its asymptotics is given by
\begin{equation}\label{asym-nse-300}
u_i(x)=\sum_{j=1}^{n}K_{ij}\tilde{b}_j+ O\bke{\frac{\cmtm{\e} x_n}{\langle x
\rangle^{n+1}}(1+1_{n=3}\log \langle x\rangle)},
\end{equation}
where $\displaystyle\tilde{b}_n=\int_{\partial B_r\cap\R^n_+} u\cdot
\nu d\sigma$, and $\displaystyle\tilde{b}_j$  for
$j<n$
is given in
\eqref{asymptotics-KMT-200} with $f$ and $F$ in
\eqref{KMT-aperture-100} and \eqref{KMT-aperture-200}.
\end{theorem}

\begin{proof} 
\cmtm{ Choose
 $r<l_1<l_2<\rho$. Taking a partition of unity for the region $B_\rho^+ \setminus B_r$, and using the pressure-independent interior and boundary estimates in \cite{Sverak-Tsai} and \cite[Theorem 3.8]{MR2027755}, we get
\[
\norm{\nb p }_{L^{n+1}(B_{l_2}^+ \setminus B_{l_1})} \le C\e .%
\]
Replacing $p$ by $p-\bar p$ where $\bar p$ is the average of $p$ in $B_{l_2}^+ \setminus B_{l_1}$, we also have $|p|<C\e$ in 
$B_{l_2}^+ \setminus B_{l_1}$ by Sobolev imbedding.
}

Recall that the cut-off $w$ defined in \eqref{eq3.51} satisfies the Navier-Stokes system \eqref{eq3.51+}--\eqref{eq3.52} in $\R^n_+$ with force $f+\nb F$ given in %
\eqref{KMT-aperture-100} and \eqref{KMT-aperture-200}.
Note $|\tilde{b}_n|\leq C\epsilon$,  both $\td w$ and $f$ have compact supports, $|\td w(x)| + |f(x)|\le C\epsilon$, and
$\abs{F(x)}\le C\epsilon\langle x\rangle^{1-a}$ with $a=n+1+\delta$ by the
hypothesis. By assumption $|w(x)| \le C\bka{x}^{-1-\si}$. 

We now first apply Theorem \ref{application-KMT-100} (ii) to get $|w(x)| \le C\bka{x}^{1-n}$, which yields the refined decay estimate
$\abs{F(x)}\le C\epsilon\langle x\rangle^{-(2n-2)}$

We next apply Theorem \ref{application-KMT-100} (i) to get $|w(x)| \le C\e \bka{x}^{1-n}$
and the asymptotic formula \eqref{asym-nse-300}.
\end{proof}

We remark that similar asymptotics as \eqref{asym-nse-300} are known
in \cite[Theorem 6.3]{BP1992} for an aperture problem in dimension
three (see also \cite[Theorem 9.1]{Galdi}). The error term presented
in \cite{BP1992} is of $O(\langle x\rangle^{-2-\eta})$ for any
$\eta\in (0,1)$ and the error term in \eqref{asym-nse-300} is
slightly better in the sense of the log correction, as well as the
presence of an anisotropic effect, namely $O\bke{\frac{ \cmtm{\e} x_3}{\langle
x\rangle^{4}}(1+\log \langle x\rangle)}$ in three dimensions.

\section{Asymptotics of fast decaying flows in the whole space and exterior domains}
\label{sec4}

In this section we study the asymptotic profiles of
\emph{fast decaying} Stokes and Navier-Stokes flows in $\R^n$ and
exterior domains.
It is well-known that the generic decay rate of these flows are $|x|^{-n+2}$.
Our concern here is flows with faster decay $|x|^{-n+1}$, usually due to some cancellation of the force.

We first choose a basis. 
For $j$, $k=1,2,\cdots, n$, we define the vector fields
\[
\Phi^{jk} = (\Phi^{jk}_1,\ldots,\Phi^{jk}_n), \quad
 \Phi^{jk}_i = \pd_k U_{ij}.
 \]
Obviously, for $n \ge 2$,
\[
| \Phi^{jk}(x) | \le C|x|^{-n+1},
\qquad
|\nb  \Phi^{jk}(x) |\le C|x|^{-n}.
\]
We will show that the asymptotic profile of a fast decaying flow
is given by the linear combination of the vector fields
$\Phi^{jk}$,  $(j,k)\neq (n,n)$.
We first collect some properties of  $\Phi^{jk}$.
\begin{lemma} Let $n \ge 2$.
 The set
 \[
 \bket{  \Phi^{jk}: \quad 1 \le  j,k \le n, \quad (j,k) \not = (n,n) }
 \]
consists of $n^2-1$ linearly independent vector fields.
\end{lemma}

\begin{proof}
We first note, by differentiating \eqref{Uij.def},
\[
\Phi^{jk}_i (x)= \pd_k U_{ij} (x)= \frac {c_n}{|x|^n} \bkt{ (\de_{ jk} - \frac {nx_j x_k}{|x|^2} ) x_i + \de_{ik} x_j - \de_{ij} x_k}.
\]
We choose $|x|=1$ and we may omit $c_n$ in the following argument.
$\Phi^{jk}$ is written as
\[
\Phi^{jk}(x) =
\begin{cases}
 (1-n x_j^2)x \qquad (k=j),
\\
 u^{jk} + v^{jk} \qquad (k \neq j),
\end{cases}
\]
$$
u^{jk} _i = - {nx_j x_k}{}  x_i, \quad v^{jk} _i = \de_{ik} x_j - \de_{ij} x_k.
$$
Note that $u^{jk} = u^{kj} =
\frac 12 (\Phi^{jk} +\Phi^{kj}) $ and
$v^{jk} = -v^{kj}= \frac 12 (\Phi^{jk} -\Phi^{kj })$ for $k \not =j$. 
Hence
\[
\text{span} \bket{ \Phi^{jk} , \Phi^{kj } }= \text{span}
\bket{ u^{jk} , v^{jk}}\quad \forall k\not=j,
\]
and
\[
\text{span} _{k \not  =j}
\bket{ \Phi^{jk}  }= \text{span}_{k < j} \bket{ u^{jk} , v^{jk}}.
\]

It is easy to see that the set $\bket{  v^{ jk}: k <j}$ contains
$\frac 12n(n-1)$ linearly independent vectors which are orthogonal to $x$.

On the other hand, the set
\[
\bket{  u^{jk}: k <j} \cup \bket{\Phi^{jj}: j<n}
\]
contains $\frac 12n(n-1) + (n-1)$  vectors which are of the form $\phi(x)  x$.
 We claim this set is linearly independent:
If
\[
f(x) := \sum_{k < j} a_{jk} x_j x_k + \sum _{l<n} b_l (1 - n x_l^2)=0,
\]
then for any $k <j$
\[
0 = \int _{|x|=1} f(x)x_j x_k = a_{jk} \int _{|x|=1} x_j^2 x_k ^2,
\]
since all other terms are odd in some variable.
Thus $a_{jk}=0$ for  any $k <j$.
We then choose $x_n=1$ and $x_j=0$ for $j<n$ to get
\[
\sum_{l<n} b_l=0.
\]
On the other hand, for fixed $m<n$ we choose $x_m=1$ and $x_j=0$ for all
$j \not = m$ to get \[
 (1-n) b_m + \sum_{l < n,~l\neq m} b_l = 0.
\]
Hence we conclude $b_m=0$ for all $m<n$.

We have shown that the set $\{\Phi^{jk}:~ (j,k)\neq (n,n)\}$ consists $n^2-1$ vector fields and the dimension of it span is  $\frac 12n(n-1) + \frac 12n(n-1) +(n-1) = n^2-1$. Thus the set is linearly independent.
\end{proof}
\noindent
\begin{remark}
\label{w^{jj}}
By the definition and the divergence free condition,
we have $ \Phi^{nn}=-\sum_{j=1}^{n-1} \Phi^{jj}$.
Therefore 
$
{\rm dim} \,\text{span} \bket{  \Phi^{jk}: 1 \le k, j \le n }
=n^2-1
$.
\end{remark}

\begin{lemma}
\label{lem:flux} Let $n \ge 2$.
 The flux $c_{jk}  :=\int_{|x|=R}  \Phi^{jk} \cdot \nu$ is zero
for every $j$, $k$. Here $\nu(x) = \frac x{|x|}$.
\end{lemma}
\noindent{\bf Proof}. Since $\div  \Phi^{jk}(x)=0$
when $x \not =0$, we have
\[
c_{jk} = \int_{|x|=R}  \Phi^{jk}(x-y) \cdot \nu(x) dS_x ,
\quad \forall |y|<R/2.
\]
Choose $\phi\in C^\I_c(\R^n)$ with support inside $B_{R/2}$
and $\int \phi =1$.
Then
\EQ{
c_{jk}
&= \int \int_{|x|=R}  \Phi^{jk}(x-y) \cdot \nu(x) dS_x \, \phi(y) dy
\\
& = -  \int_{|x|=R} \int  \pd_{y_k} U_{ij}(x-y) \phi(y) dy\,\nu_i(x) dS_x
\\
& =  \int_{|x|=R} \int   U_{ij}(x-y) \pd_{y_k} \phi(y) dy\,\nu_i(x) dS_x
\\
& =  \int  \bke{\int_{|x|=R}   U_{ij}(x-y) \nu_i(x) dS_x }\pd_{y_k} \phi(y) dy\, = 0.
}

We now consider Stokes and Navier-Stokes flows in the whole space and exterior domains.
A \emph{weak solution} of  the Stokes system \eqref{eq1-1} (or of the Navier-Stokes system \eqref{eq1-2})   in $\Om \subset \R^n$ is a vector field $v \in  W^{1,2}_{loc}(\Om)$ that satisfied the weak form of \eqref{eq1-1} (or of \eqref{eq1-2}) with divergence-free test functions, with no assumption on its global integrability nor its boundary value in this section.

\begin{lemma}[Uniqueness in $\R^n$]\label{th:uniquenessRn} Let $v \in W^{1,2}_{loc}({\R^n})$, $n \ge 2$, be a weak solution of the Stokes system  \eqref{eq1-1}   in $ \R^n$  with zero force. If $v(x)=o(|x|)$ as $|x|\to \infty$, then $v $ is constant. 
\end{lemma}
\begin{proof} By the pressure independent estimates of \cite{Sverak-Tsai} and bootstraping, $v$ is locally $C^1$ and
\[
\norm{\nb v}_{L^\infty(B_R)} \le \frac CR \norm{ v}_{L^\infty(B_{2R})}.
\]
Taking $R \to \infty$, we get $\nb v\equiv 0$. 
\end{proof}

\begin{proposition} [Asymptotics of the Stokes flows
in the whole space] \label{thm:whole1} Let $n \ge 2$. Let $v \in
H^1_{loc}(\R^n)$ be a weak solution of
\[
-\Delta v+\nabla p=f,
\qquad{\rm
div}\,v=0\qquad \mbox{in}\,\,\R^n
\]
with $f$ satisfying, for some $a >n+1$,
\begin{equation}
\label{cancel}
|f(x)|\lec \langle x \rangle^{-a}, \quad
\int_{\R^n}f(x)dx=0.
\end{equation}
Assume that  $v$ satisfies
\[
|v(x)|\lec 
\cmtm{ o(1) \quad \text{as } |x| \to \infty.}
\label{est:td}
\]
Then $|v(x)|\lec  \langle x\rangle^{-n+1}$, and its asymptotics is given as
\begin{equation}
\label{est:whole}
v_i(x)=
\sum_{(j,k)\neq (n,n)} \Phi^{jk}_i(x)b_{jk}
+
O(\de(x) ) \qquad (|x|>1),
\end{equation}
where $\de(x) = \bka{x}^{-\min\{n, a-2\}} \cmtm{(1+ 1_{a=n+2} \log \bka{x})}$,
\begin{equation}
\label{bjk.def}
b_{jk}=-\int_{\R^n}y_k f_j(y)dy\quad\text{for}\quad j\neq k,
\quad
b_{jj}=\int_{\R^n}(y_n f_n(y) - y_j f_j(y))dy.
\end{equation}
\end{proposition}

\begin{proof}
By uniqueness in the class  \eqref{est:td} using Lemma \ref{th:uniquenessRn},
\[
v_i(x)=\int_{\R^n} U_{ij}(x-y)f_j(y)dy.
\]
By \eqref{cancel}, 
\begin{align}
v_i(x)
&= \int_{\R^n} (U_{ij}(x-y)-U_{ij}(x))f_j(y)dy
\notag
\\
&= 
\sum_{j,k=1}^n  \Phi^{jk}_i(x) \hat b_{jk} 
+\sum_{j,k=1}^n%
\int_{\R^n} (U_{ij}(x-y)-U_{ij}(x)+ \Phi^{jk}_i(x)y_k)f_j(y)dy,
\label{eq:expansion}
\end{align}
where
$\hat b_{jk}=-\int_{\R^n}y_k f_j(y)dy$ for $j$, $k=1,2,\cdots, n$.
By Remark \ref{w^{jj}} and the definition of $b_{jk}$, %
\begin{align*}
\sum_{j,k=1}^n  \Phi^{jk}(x) \hat b_{jk}
=
\sum_{(j,k)\neq (n,n)}\Phi^{jk}(x) b_{jk}.
\end{align*}

The second term in \eqref{eq:expansion} is the error. To estimate it, we may assume $|x|>2$. We split it as
$$
\int_{\R^n} (U_{ij}(x-y)-U_{ij}(x)+ \Phi^{jk}_i(x) y_k)f_j(y)dy
=\int_{|y|\le |x|/2}
+
\int_{|y|>|x|/2}
=
I+I\!I.
$$
By the Taylor theorem and the estimate
$
|\pd_{kl}^2 U_{ij}(x)|\lec |x|^{-n},
$
we get for $\theta =\theta(x,y)\in [0,1]$
\[
|I |
=
|\int_{|y|\le |x|/2}\pd^2_{kl}U_{ij}(x-\theta y)y_ky_lf_j(y)dy|
\lec
|x|^{-n} \int_{|y|\le |x|/2}\langle y\rangle^{-a+2}dy.
\]
By \eqref{eq3.15},
\[
|I|\lec |x|^{-\min(n,a-2)} (1+ 1_{a=n+2} \log |x|).
\]
For $I\!I$,
\EQ{
|I\!I | 
\lec \int_{|y|>|x|/2}
\bke{ |x-y|^{2-n} + |x|^{2-n} +   |x|^{1-n}  |y|} |y|^{-a}
dy
 = C |x|^{-a+2}.
}
The last equality is by scaling.
The proof is complete.
\end{proof}

We next consider the Stokes flows in exterior domains.

\begin{proposition} [Asymptotics for the exterior Stokes flows]
\label{thm:exterior1}
Let $n \ge 2$ and $\Omega \subset \R^n$ be an exterior Lipschitz domain  
 with $0 \not \in \wbar \Om$.
Assume $(v,p) \in H^2_{loc}(\wbar \Omega) \times H^1_{loc}(\overline \Omega)$
is a  solution of
\[
-\Delta v+\nabla p=f,
\qquad{\rm
div}\,v=0\qquad \mbox{in }\,\,
\Omega,
\]
satisfying
\begin{equation}
|f(x)|\lec \langle x \rangle^{-a}
\qquad {\it with}\ a >n+1,
\end{equation}
\begin{equation}
\label{cancel2}
\cmtm{ 
\int_{\Omega} f +\int_{\pd \Omega} (\nu \cdot \nb v  -p\nu)=0,}
\end{equation}
and
\begin{equation}
\label{est:up}
|v(x)|\lec  \cmtm{o(1),\quad \text{as } |x|\to \infty.} %
\end{equation}
Then its asymptotics is given as
\begin{equation}
\label{est:exterior1}
v_i(x)=\tilde{b}_0 H_i(x)+\sum_{(j,k)\neq (n,n)}  \Phi^{jk}_i(x)
\tilde{b}_{jk}+
O(\de(x) ),
\end{equation}
where $\de(x) = \bka{x}^{-\min\{n, a-2\}} \cmtm{(1+ 1_{a=n+2} \log \bka{x})}$,
\begin{equation}
\tilde{b}_0=\int_{\pd {\Omega}} v\cdot \nu dS, \qquad
H(x) =\nabla E(x) = \frac{-x}{n \om_n |x|^n},
\end{equation}
and
\begin{equation}
\label{b}
\tilde{b}_{jk}=
\begin{cases}
-\int_{\Omega}y_kf_jdy
\cmtm{+} \int_{\pd \Omega}
(v_j\nu_k-y_k\nabla v_j\cdot \nu+y_kp \nu_j)dS
\hfill \textit{if} \ j\neq k,
\\
 \begin{split}
\int_{\Omega}(y_nf_n - y_jf_j )dy
\cmtm{+} \int_{\pd \Omega} \left \{(v_j\nu_j-y_j\nabla v_j\cdot \nu+y_jp \nu_j) \right \}dS
\\
\cmtm{-}
 \int_{\pd \Omega}
\left \{(v_n\nu_n-y_n\nabla v_n\cdot \nu+y_n p \nu_n)\right \}dS
\end{split}
\quad
 \textit{if} \ j=k.
\end{cases}
\end{equation}
\end{proposition}

\begin{remark} (i) From the proof of Lemma \ref{lem:flux},
we see that $H$ is linearly independent of
the vectors $ \Phi^{jk}$ for $1\le j$, $k \le n$. 

(ii) If we restrict a solution $(v,p)$ of Proposition \ref{thm:whole1} to $\Om$, since 
\[
\int_{\Om^C} f=
\int_{\Om^C} - \De v_i + \pd_i p = \int_{\pd \Om} \pd_j v_i \nu_j - p \nu_i,
\]
we get condition \eqref{cancel2} from $\int_{\R^n} f=0$.
\end{remark}

\begin{proof} 
First note that, by replacing $v $ by $\td v= v- \tilde{b}_0 H$,
we may assume $\int_{\pd \Omega} v\cdot \nu=0$. \cmtm{Note that \eqref{cancel2} 
and \eqref{b} are not changed by this replacement because, for \eqref{cancel2},
\[
\int_{\pd \Om} \nu \cdot \nb H_i = \int_{ \Om\cap B_R} \div \nb H_i - \int_{\pd B_R} \frac xR \cdot \nb H_i = O(1/R),
\]
which vanishes as $R\to \infty$;
For \eqref{b} and $R > \text{diam}\,(\Om)$,
\EQ{
\int_{\pd \Om} &\bket{H_j \nu_k - y_k \pd_l H_j \nu_l} = \int_{\pd \Om} \bket{2H_j \nu_k-\pd_l( y_k H_j) \nu_l }
\\
&= 
\int_{\Om \cap B_R} \bket{2\pd_k H_j  - \De ( y_k H_j) } - \int_{\pd B_R} \bket{H_j \nu_k - y_k \pd_l H_j \nu_l} =0+ \de_{kj}.
}}
 
Let $\Om_1$ be any exterior domain with $\wbar \Om_1 \subset \Om$, and
$\chi$ be any smooth function with $\supp \chi \subset \Omega$ and
$\chi= 1$ in $\Om_1$. 
We define $(w, q)$ by
\begin{equation}
\label{eq:v}
w=v\chi +\hat{v}, \qquad q=p\chi,
\end{equation}
where $\hat{v}$ is a solution of
$\div\, \hat{v}=-v\cdot \nabla \chi$ in $\R^n$.
Thanks to the condition $\int_{\pd \Omega}v\cdot \nu=0$,
we can choose $\hat{v}$
satisfying
$\supp \hat{v} \subset \wbar{\Omega} \setminus \Omega_1$,
$\|\nabla \hat{v}\|_{L^s} \lec_s \|v \cdot \nabla \chi\|_{L^s}$
by \cite[Theorem III.3.1]{Galdi}.  Then $(w,q)$ satisfies
\begin{align*}
-\Delta w +\nabla q=g, \qquad \div w=0 \qquad \textrm{in} \ \R^n,
\\
g=f\chi-\pd_l(v \pd_l\chi)-\pd_l v \pd_l\chi
-\Delta\hat{v} + p\nabla \chi.
\end{align*}

From the assumption for $f$, we easily see that $|g(x)| \lec \langle x \rangle^{-a}$ and
$\int_{\R^n} gdx=0$ because
\begin{align*}
\int_{\R^n}gdx 
&=\int_{\Omega}\{f\chi-\pd_l(v \pd_l\chi)-\pd_l v \pd_l(\chi-1)
-\Delta\hat{v} + p\nabla (\chi-1)\}  
\\
&=\int_{\Omega}f\chi 
- \int_{\pd \Omega} v \pd_l\chi\nu_l 
+\int_{\Omega} \De v (\chi-1)+\int_{\pd \Omega} \pd_l v \nu_l 
-\int_{\Omega} \nabla p (\chi-1)  -\int_{\pd \Omega} p\nu 
\\
&=\int_{\Omega} f +\int_{\pd \Omega} (\nu \cdot \nb v  -p\nu),
\end{align*}
which is zero by assumption \eqref{cancel2}.
Then Proposition \ref{thm:whole1} shows
 $w_i=  \Phi^{jk}_i b_{jk}+R_i$ with
$|R_i(x)| \le C \de(x)$
and
\begin{equation}
b_{jk}=
\begin{cases}
-\int_{\R^n} y_k g_j(y)dy \qquad\textrm{for}\ j\neq k,
\\
-\int_{\R^n} (y_k g_j(y)-y_n g_n(y))dy
\qquad \textrm{for}\ j=k.
\end{cases}
\label{eq:b_{jk}}
\end{equation}
We claim that $b_{jk}$ is independent of the choice of
the cut-off \eqref{eq:v}.
Indeed, let $\chi'$, $(w',q')$ be another cut-off solution of \eqref{eq:v}, 
and $b'_{jk}$ be as in \eqref{eq:b_{jk}},
 then for $j \neq k$,
\begin{align*}
b_{jk}-b'_{jk}&=
-\int_{\R^n}
y_k \{
-\Delta (v_j(\chi-\chi')+\hat{v}_j-\hat{v}_j')
+\partial_j (p(\chi-\chi'))\}
\\
&=\int_{\R^n}
\partial_j y_k (p(\chi-\chi'))
=0.
\end{align*}
Similary, 
\begin{align*}
b_{jj}-b'_{jj}
&=
\int_{\R^n}
\left\{\partial_j y_j (p(\chi-\chi'))
-\partial_n y_n (p(\chi-\chi'))\right\}
=0.
\end{align*}
Thus the claim follows. 

Now consider a sequence of cut-off functions $\chi_m$, $m=1,2,\cdots$,
such that $0 \le \chi_m(x) \le \chi_{m+1}(x) \le 1$, 
($m=1,2,\cdots$) and $\chi_m(x) \to 1$ as $m \rightarrow \infty$ for all $x \in \Om$, 
and choose $(w^{(m)},q^{(m)})$, $b^{(m)}_{jk}$ as in
\eqref{eq:v}, \eqref{eq:b_{jk}}.
Since $b^{(m)}_{jk}$ is  independent of $m$, it suffices to
prove $\lim_{m\rightarrow \infty}b^{(m)}_{jk}=\tilde{b}_{jk}$.
We only consider the case $j\neq k$, since
the case $j=k$ is shown in the same way.
We divide
\begin{align*}
b^{(m)}_{jk}
&=
-\int_{\R^n}
y_k\bket{ f_j\chi-\pd_l(v_j \pd_l\chi)-\pd_l v_j \pd_l\chi
- \Delta\hat{v}_j
+p\pd_j \chi} dy
\\
&=I+I\!I+I\!I\!I+I\!V+V.
\end{align*}
Then it easily follows that
$I \rightarrow -\int_{\Omega}y_kf_jdy$
as $m\rightarrow \infty$ and that $I\!V=0$.
By integration by parts, we also observe
\begin{align*}
I\!I
&=
-\int_{\R^n}
\pd_l(y_k)v_j \pd_l\chi dy
=
-\int_{\Omega}
v_j \pd_k(\chi-1)dy
\\
&=
\int_{\Omega}\pd_k v_j(\chi-1)
\cmtm{+}
\int_{\pd \Omega} v_j\nu_k
\rightarrow \int_{\pd \Omega} v_j \nu_k,
\\
I\!I\!I
&=
- \int_{\Omega}
\pd_l(y_k \pd_j v_l) (\chi-1)dy
-
\int_{\pd \Omega} y_k \pd_l v_j  \nu_l
\rightarrow -
\int_{\pd \Omega} y_k \pd_l v_j  \nu_l,
\\
V &=
 \int_{\Omega}
\pd_j(y_k p) (\chi-1)dy
+\int_{\pd \Omega} y_kp\nu_j
\rightarrow \int_{\pd \Omega} y_kp\nu_j,
\end{align*}
as $m \rightarrow \infty$, using $\chi \to 1$ in $\Omega$.
Thus we have proved \eqref{b}.
\end{proof}

\begin{theorem} [Asymptotics of fast decaying Navier-Stokes flows in $\R^n$]
\label{thm:whole2}
Let $n\ge 3$, and $u\in H^1_{loc}(\R^n)$ be a weak solution of
\[
\label{eq:NSE}
-\Delta u+(u\cdot\nabla )u+\nabla p=f,
\qquad{\rm
div}\,u=0\qquad \mbox{in }\,\,\R^n.
\]
There exists $\ve_0 >0$ such that
if for some $\e\in (0,\ve_0]$,
\[
|f(x)|\le \ve \langle x \rangle^{-a}\quad \text{with }
a>n+1,\quad
\int_{\R^n}f(x)dx=0,
\]
 and
\begin{equation}
\label{est:u}
|u(x)|\le \ve \langle x\rangle^{1-n},
\end{equation}
then 
 its asymptotics is given as
\begin{equation}
\label{est:NSE}
u_i(x)=\sum_{(k,j)\neq(n,n)}  \Phi^{jk}_i(x)
a_{jk}+
O(\e \de(x)),
\end{equation}
where $\de(x) = \bka{x}^{-\min\{a-2,\,n\}}(1+1_{a=n+2}\log \bka{x})$,
$$
a_{jk}=
\begin{cases}
b_{jk} \cmtm{-} \int_{\R^n}u_ju_kdy \qquad \textrm{for} \ j\neq k,
\\
b_{jj} \cmtm{-} \int_{\R^n}(u_j^2-u_n^2)dy \qquad \textrm{for}\ j=k.
\end{cases}
$$
Here $b_{ij}$ are the constants given by \eqref{bjk.def}.
\end{theorem}

\emph{Remark.} We can replace \eqref{est:u} by a weaker condition $|u(x)| \le \e \bka{x}^{-1-\si}$, $\si>0$: Under this weaker condition, we can improve the decay iteratively, $|u(x)| \lec \e \bka{x}^{-1-(1+k /2)\si}$, $k \in \N$, as in the proof of Theorem \ref{application-KMT-100}, Case (ii).

\begin{proof}
By the scaling and the bootstrapping argument
as in \cite{Sverak-Tsai}, we see
$|\nabla u(x)|\lec \ve \langle x\rangle^{-n}$.
Indeed, if $R=\frac 12 |x_0|>2$, let $v(y) = R^{n-1} u(x)$, $\pi(y) = R^n p(x)$ and $g(y) = R^{n+1} f(x)$ with $x=x_0+Ry$. Then $v$ satisfies
\[
- \De v + R^{2-n} v \cdot \nb v + \nb \pi = g
\]
with $|v| \lec \ve$ and $g \lec \ve$ in $B_1$. By  bootstrapping (using the pressure-independent Stokes estimate of  \cite{Sverak-Tsai}), one gets $|\nabla v| \lec \ve$ in $B_{1/2}$, which implies $|\nabla u(x_0)|
\lec \ve R^{-n}$.

Note that $u$ is the solution of the Stokes equations
with force $\tilde{f}=f-u\cdot \nabla u$ satisfying
$
|\tilde{f}(x)| \lec \langle x \rangle^{-\min\{a, 2n-1 \}}
$.
Since $\min\{a,\,2n-1 \} >n+1$ and $2n-1>n+2$ using $n \ge 3$,
it follows from Proposition \ref{thm:whole1} that
$$
u_i(x)=\sum_{(j,k)\neq(n,n)} \Phi^{jk}_i(x)
a_{jk}+
O(\de(x)),
$$
where 
$$
a_{jk}=
\begin{cases}
-\int_{\R^n} y_k(f_j-u\cdot \nabla u_j)(y) dy \qquad
\textrm{for} \ j \neq k,
\\
-\int_{\R^n} (y_jf_j-y_nf_n)
-(y_j u\cdot \nabla u_j-y_n u\cdot \nabla u_n) dy
\qquad
\textrm{for} \ j=k.
\end{cases}
$$
Then noting that $\int_{\R^n} y_k(u\cdot \nabla u_j)(y) dy
=-\int_{\R^n}u_ju_k(y)dy $, we obtain the desired result.
\end{proof}

\begin{theorem} [Asymptotics of  fast decaying  exterior Navier-Stokes flows]
\label{thm:exterior2} Let $n \ge 3$ and $\Omega \subset \R^n$ be
an exterior Lipschitz  domain with $0 \not \in \wbar \Om$, and let $(u,p)\in H^2_{loc}(\wbar \Omega) \times
H^1_{loc}(\wbar \Omega)$ be a solution of
\[
-\Delta u+(u\cdot\nabla )u+\nabla p=f,
\qquad{\rm
div}\,u=0\qquad \mbox{in }\,\,\Omega.
\]
There exists $\ve_0 >0$ such that
if for some $\e\in (0,\ve_0]$,
$|f(x)|\lec \ve\langle x \rangle^{-a}$ with
$a>n+1$,
\[
|u(x)|\lec \ve \langle x\rangle^{-n+1},
\]
\begin{equation}
\label{cancel3}
\int_{\Omega}  f +\int_{\pd \Omega} (\nu \cdot \nb u  -p\nu - (u \cdot \nu)u)=0.
\end{equation}
Then its asymptotics is given as
\begin{equation}
\label{est:exterior2}
u_i(x)=\tilde{b}_0H_i(x)
+\sum_{(j,k)\neq(n,n)} \Phi^{jk}_i(x)
\tilde{a}_{jk}+
O(\e \de(x)),
\end{equation}
where $\de(x) = \bka{x}^{-\min\{a-2,\,n\}}(1+1_{a=n+2}\log \bka{x})$,
$$\tilde{a}_{jk}=
\begin{cases}\tilde{b}_{jk}\cmtm{-} \int_{{\Omega}}u_ju_kdy
+ \cmtm{ \int_{\pd {\Omega}} y_ku_ju\cdot \nu dS}
\qquad
\textrm{for} \ j\neq k,
\\
\tilde{b}_{jj}  \cmtm{-} \int_{{\Omega}}(u_j^2-u_n^2)dy
+ \cmtm{ \int_{\pd {\Omega}} (y_ju_j- y_n u_n) u\cdot \nu dS}
\qquad \textrm{for} \ j=k.
\end{cases}
$$
Above  $\tilde{b}_0$, $\tilde{b}_{jk}$ and $H(x)$ are defined in
Proposition \ref{thm:exterior1}.
\end{theorem}

\begin{remark} If we restrict a solution $(u,p)$ of Theorem \ref{thm:whole2} to $\Om$, since 
\[
\int_{\Om^C} f=
\int_{\Om^C} - \De u_i + \pd_j( u_j u_i ) + \pd_i p = \int_{\pd \Om} (\pd_j v_i - u_ju_i) \nu_j - p \nu_i,
\]
we get condition \eqref{cancel3} from $\int_{\R^n} f=0$.
\end{remark}

\begin{proof}
As in the proof of Theorem \ref{thm:whole2}, $(u, p)$ satisfies
the linear Stokes system with
the force $\tilde{f}=f-(u\cdot \nabla u)$ in $\Omega$
and $|\nabla u(x)|\lec \e |x|^{-n}$.
Here $|\tilde{f} (x)| \le \e \langle x\rangle^{-\min \{ a, 2n-1\}}$
with $\min \{ a, 2n-1\}>n+1$.
\eqref{cancel3} and the integration by parts yield
\begin{equation}
\int_{\Omega} \tilde f +\int_{\pd \Omega} (\nu \cdot \nb u  -p\nu)=0.
\end{equation}
Then Proposition \ref{thm:exterior1} shows
$$
u(x)=\tilde{b}_0 H(x) + \sum_{k=1}^{n}\Phi^{jk}(x)
\tilde{a}_{jk} +
O(\e \de(x)).
$$
Here we have for $j\neq k$ that
\EQ{
\tilde{a}_{jk}&
=
-\int_{{\Omega}}y_k( f_j(y)- u\cdot \nabla u_j(y))dy
+ \int_{\pd {\Omega}}
(u_j\nu_k- y_k\nabla u_j\cdot \nu+ y_kp \nu_j)dS
\\
&=-\int_{{\Omega}}(y_k f_j(y)+u_ju_k(y))dy
+ \int_{\pd {\Omega}}
(u_j\nu_k - y_k\nabla u_j\cdot \nu + y_kp \nu_j +y_ku_ju\cdot \nu)dS.
}
The case $j=k$ is handled in the same way.
Hence the proof is complete.
\end{proof}
\section{Appendix: Axisymmetric self-similar solutions in $\R^3_+$}
 \label{sec5}

In this appendix
we consider the nonexistence of minus one homogeneous solutions of the steady-state
Navier-Stokes equations in the half-space $\R^3_+$ with the \emph{Navier} boundary
conditions (BC),
\begin{equation}\label{axially-KMT-10}
-\Delta u+(u\cdot\nabla )u+\nabla p=0,\qquad \mbox{in }\,\R^3_+,
\end{equation}
\begin{equation}\label{axially-KMT-20}
u \cdot \nu = 0, \quad \bke{(1-\gamma)\frac{\pd u}{ \pd \nu }+ \ga
u} \times \nu=0,\qquad \mbox{on }\,\partial \R^3_+\setminus\{0\},
\end{equation}
for some given  $\gamma\in[0,1]$, and $\nu$ is the unit outernormal, $\nu=(0,0,-1)$ for $\R^3_+$.
Note that the Navier BC becomes the zero Dirichlet (no-slip) BC if $\ga=1$, which is what we used in Section \ref{sec3}. When $\ga=0$, it agrees with the {\em slip BC} for a half space, see e.g.~\cite{Xiao-Xin}. Their nonexistence excludes an obstacle for proving the asymptotic results in Section \ref{sec3} which have faster decay, even under the more general Navier BC. 
Recall that in the whole space we have the family of Slezkin-Landau solutions.

\begin{theorem}
Let $u$ be a minus one homogeneous solution of the Navier-Stokes
equations \eqref{axially-KMT-10} in $\R^3_+$ with the Navier BC \eqref{axially-KMT-20} for some $\gamma\in[0,1]$. If $u$ is
axially symmetric with respect to $x_3$-axis, then $u$ vanishes.
\end{theorem}

{\it Remark.}\quad Note $u$ is allowed to have nonzero $u_\th$-component. If we do not impose any BC, then the restrictions of the Slezkin-Landau solutions are non-trivial solutions.

The following proof is adapted from the corresponding argument of Tsai and Sverak for the whole space \cite[Section 4.3]{Tsai-thesis}. 

\begin{proof}
We use the spherical coordinate $(\rho, \theta, \varphi)$, where
$\rho=\abs{x}$, $\theta$ is azimuthal angle, and $\varphi$ is
between the angle $x$ and $x_3$-axis. The axially symmetric solution
is of the form
\begin{equation}\label{axially-KMT-21}
u= \frac 1\rho f(\ph) e_\rho +  \frac 1\rho g(\ph)
e_\ph+\frac{1}{\rho} h(\ph)e_{\theta}.
\end{equation}
We note that the boundary conditions \eqref{axially-KMT-20} becomes
\begin{equation}\label{axially-KMT-23}
g(\frac \pi 2)=0, \qquad (1-\ga) f' (\frac \pi 2)+\ga f (\frac \pi
2)=0, \qquad (1-\ga) h' (\frac \pi 2)+\ga h (\frac \pi 2)=0.
\end{equation}
We also observe, due to symmetry, that
\begin{equation}\label{axially-KMT-27}
f'(0)=g(0)=h(0)=0.
\end{equation}
The equations can be rewritten in spherical coordinates as follows:
\begin{equation}\label{axially-KMT-30}
f^{''}+f'\cot\varphi=gf'-(f^2+g^2)-2p,
\end{equation}
\begin{equation}\label{axially-KMT-40}
f'=gg'+p',
\end{equation}
\begin{equation}\label{axially-KMT-310}
(h'+h\cot\varphi)'=g(h'+h\cot\varphi).
\end{equation}
\begin{equation}\label{axially-KMT-50}
f+g'+g\cot\varphi=0.
\end{equation}
Setting $H(\varphi):=h'+h\cot\varphi=(h\sin\varphi)'/\sin\varphi$,
we see that \eqref{axially-KMT-310} is rewritten as $H'=gH$. We
claim that $H=0$, which obviously implies $h=0$, due to boundary
conditions \eqref{axially-KMT-300}. We treat the cases of
$\gamma=0$, $0<\gamma<1$ and $\gamma=1$, separately. We note first
that if $H$ has a zero at a point in $[0,\pi/2]$, it vanishes
everywhere due to uniqueness of ODE. In case that $\gamma=0$, it is
direct via \eqref{axially-KMT-23} that $H(\pi/2)=0$, and thus $H=0$.
In case that $0<\gamma<1$, we note that
\begin{equation}\label{navier-boundary-10}
H(\frac{\pi}{2})=h'(\frac{\pi}{2})=-\frac{\gamma}{1-\gamma}h(\frac{\pi}{2}).
\end{equation}
On the other hand, we also observe that
\begin{equation}\label{navier-boundary-20}
\int_0^{\frac{\pi}{2}} H(\varphi)\sin\varphi
d\varphi=h(\frac{\pi}{2}).
\end{equation}
Suppose that $H$ has no zero, which means that $H$ is either
positive or negative on $[0, \pi/2]$. If $H$ is positive, then
$h(\pi/2)<0$ via \eqref{navier-boundary-10}. Then, it is contrary to
\eqref{navier-boundary-20}. The other case that $H$ is negative also
lead to a contradiction. Thus, $H=0$. Finally, for the case
$\gamma=0$, we have due to \eqref{axially-KMT-300}
\[
\int_0^{\frac{\pi}{2}} H(\varphi)\sin\varphi
d\varphi=h(\frac{\pi}{2})=0.
\]
This implies that $H$ has a zero. Therefore, we conclude that $H$
vanishes, and so $h=0$.

Integrating \eqref{axially-KMT-40}, we have
\begin{equation}\label{axially-KMT-60}
f=\frac{g^2}{2}+p+C_1
\end{equation}
for some constant $C_1$. Combining \eqref{axially-KMT-30} and
\eqref{axially-KMT-60}, we obtain
\begin{equation}\label{axially-KMT-70}
f^{''}+f'\cot\varphi=gf'-f^2-2f+2C_1.
\end{equation}
We set $A:=f(\varphi)\sin\varphi$ and $B:=g(\varphi)\sin\varphi$.
Noting that $-A=B'$, we see that \eqref{axially-KMT-70} becomes
\begin{equation}\label{axially-KMT-80}
(f^{'}\sin\varphi)'=(Bf)'+2B'+2C_1\sin\varphi.
\end{equation}
Therefore, we obtain
\begin{equation}\label{axially-KMT-90}
f^{'}\sin\varphi=Bf+2B-2C_1\cos\varphi+C_2.
\end{equation}
Via \eqref{axially-KMT-27}, we see that $C_2=2C_1$ and thus,
\begin{equation}\label{axially-KMT-100}
f^{'}\sin\varphi=Bf+2B-2C_1(\cos\varphi-1).
\end{equation}
Let $L(t)=B(\varphi)$ with $t=\cos\varphi$. Noting that
$B'=-L'\sin\varphi$ and $B^{''}=L^{''}\sin^2\varphi-L'\cos\varphi$,
we observe that $L$ satisfies
\[
(1-t^2)L^{''}+2L+LL'=2C_1(t-1),\qquad L(0)=L(1)=0.
\]
Using the change of variable $L(t):=(1-t^2)v(t)$, we see that $v$
solves
\[
\bke{(1-t^2)^2v^{'}}'+\bke{\frac{(1-t^2)^2v^{2}}{2}}'-2C_1(t-1)=0,\qquad
v(0)=0,
\]
which can be simplified as follows:
\[
v'+\frac{v^2}{2}=\frac{C_1}{(1+t)^2}+\frac{C_3}{(1-t^2)^2},\qquad
v(0)=0.
\]
Since $v$ is bounded over $t\in [0,1]$, we see that $C_3=0$, and
therefore, $v$ satisfies
\[
v'+\frac{v^2}{2}=\frac{C_1}{(1+t)^2},\qquad v(0)=0.
\]
In addition, we note that
\[
v'(0)=- g'(\frac \pi 2) = C_1, \qquad  v^{''}(0)=g^{''}(\frac \pi
2)= -2C_1,
\]
Recalling that $f = -g' - g \cot \ph$ and by taking one more
derivative, we also see that $f' = -g^{''} - g'\cot \ph
+\frac{g}{\sin^2\varphi}$. Therefore, we obtain
\begin{equation}\label{axially-KMT-200}
 f(\frac \pi 2)= - g'(\frac \pi 2),
\qquad f'(\frac \pi 2)= - g''(\frac \pi 2).
\end{equation}
Combining \eqref{axially-KMT-23} and \eqref{axially-KMT-200}, we see
that
\begin{equation}\label{axially-KMT-300}
0=(1-\ga) f' (\frac \pi 2)+\ga f (\frac \pi 2)=2(1-\ga) C_1 +\ga
C_1=(2-\ga)C_1.
\end{equation}
Since $\gamma\in [0,1]$, we conclude that $C_1=0$, which implies
$v=0$. Then, it is straightforward that $f=g=0$, which implies that
$u$ vanishes. This completes the proof.
\end{proof}

\section*{Acknowledgments}
We thank the  Center for Advanced Study in Theoretical Sciences at the National Taiwan University, where part of this project was
conducted when all of us visited in March 2014 and in July 2015.
The research of Kang was partially
supported by NRF-2014R1A2A1A11051161 and NRF-20151009350.
The research of Miura was partially
supported by JSPS grant 25707005.
The research of
Tsai was partially supported by NSERC grant 261356-13.

\bibliographystyle{abbrv}
\bibliography{kmt-bib}
\end{document}